\documentclass [10pt, article, oneside, a4paper]{memoir}

\usepackage{amsmath,amsfonts,amssymb,amsthm,paralist,graphicx,xargs}
\usepackage[all]{xy}
\usepackage[backend=biber,maxbibnames=15]{biblatex}
\usepackage[colorlinks=true,linkcolor=black,citecolor=blue]{hyperref}
\usepackage{tikz}
\usetikzlibrary{bending,decorations.markings}
\newcommand{\aquatre}{
\stockaiv\pageaiv
\settypeblocksize{22cm}{13cm}{*}
\setlrmargins{4cm}{*}{1}
\setmarginnotes{0pt}{0pt}{0pt}
\setulmargins{3.5cm}{*}{1}
\setheadfoot{3\baselineskip}{3\baselineskip}
\setheaderspaces{2\baselineskip}{*}{1}
\checkandfixthelayout
}
\aquatre

\newcommand{\addpoint}[1]{#1\ ---\ }

\setsecnumdepth{subsubsection}

\setsecnumformat{\csname the#1\endcsname---}
\setsubsubsechook{\setsecnumformat{{\normalfont\sffamily\csname
the##1\endcsname\ }}}
\setsubsechook{\setsecnumformat{\csname the##1\endcsname\ ---\ }}
\setsechook{\setsecnumformat{\csname the##1\endcsname---}}

\setsecheadstyle{\centering\Large\bfseries\sffamily}
\setsubsecheadstyle{\large\bfseries\sffamily}
\setsubsubsecheadstyle{\bfseries\itshape\addpoint}
\setsubsubsecindent{1em}
\setbeforesubsubsecskip{.5em plus .2em minus -.1em}
\setaftersubsubsecskip{-0em}
\setsubparaheadstyle{\itshape}
\newtheoremstyle{thm}
{1.5ex plus .3ex minus .1ex}
{1ex plus .3ex minus .1ex}
{\itshape}
{}
{\sffamily}
{---}
{0em}
{$\bullet$\hbox{\ }#1\hbox{\ }#2}

\theoremstyle{thm}

\newtheorem{definition}{Definition}[section]

\newtheorem{theorem}[definition]{Theorem}

\newtheorem{lemma}[definition]{Lemma}
\newtheorem{proposition}[definition]{Proposition}
\newtheorem{corollary}[definition]{Corollary}

\newtheoremstyle{note}
{1ex plus .3ex minus .1ex}
{1ex plus .3ex minus .1ex}
{}
{}
{\itshape}
{.}
{1em}
{}
\theoremstyle{note}
\newtheorem{example}[definition]{Example}
\newtheorem{remark}[definition]{Remark}
\newtheorem{discussion}[definition]{Discussion}

\newtheoremstyle{q}
{1ex plus .3ex minus .1ex}
{1ex plus .3ex minus .1ex}
{}
{}
{\bfseries}
{:}
{1em}
{}
\theoremstyle{q}

\newcommand{\M}{\mathcal {M}}
\newcommand{\R}{\mathcal {R}}
\newcommand{\W}{\mathcal {W}}
\newcommand{\I}{\mathcal {I}}

\newcommand{\len}[1]{|#1|}
\newcommand{\up}[1]{\uparrow\! #1}
\newcommand{\Up}[1]{\Uparrow\! #1}

\newcommand{\RR}{\mathbb{R}}

\newcommand{\F}{\mathcal{F}}
\newcommand{\G}{\mathcal{G}}
\newcommand{\T}{\mathcal{T}}
\renewcommand{\H}{\mathcal{H}}
\newcommand{\Fbar}{\overline{\F}}
\newcommand{\Cbar}{\overline{C}}
\renewcommand{\SS}{\mathcal{S}}
\newcommand{\ftilde}{\widetilde{f}}
\newcommand{\gtilde}{\widetilde{g}}
\newcommand{\htilde}{\widetilde{h}}
\newcommand{\ve}{\varepsilon}
\newcommand{\tq}{\,:\,}

\newcommand{\un}[1]{\mathbf{1}_{\{#1\}}}
\newcommand{\Mbar}{\overline\M}
\newcommand{\BM}{\partial\M}
\newcommand{\DSC}{\text{\normalfont\textsf{DSC}}}
\newcommand{\DSCP}{\ensuremath{\DSC^+}}

\newcommand{\C}{\mathcal{C}}
\newcommand{\Cstar}{\mathfrak{C}}

\DeclareMathOperator{\Id}{Id}
\newcommandx*{\seq}[3][3=0]{(#1_{#2})_{#2\geqslant#3}}
\newcommand{\slgb}{$\sigma$-algebra}
\newcommand{\myrho}[2]{\rho[#1,#2]}
\newcommand{\mymu}[2]{\mu[#1,#2]}
\newcommand{\Mob}{\mathsf{M}}
\newcommand{\iid}{\textsl{i.i.d.}}
\newcommand{\as}{\text{a.s.}}
\newcommand{\esp}{\mathbb{E}}

\newcommand{\height}[1]{\tau(#1)}

\newcommand{\ie}{\textsl{i.e.}}

\begin{document}

\begin{center}
{\huge\bfseries\sffamily
Ergodic properties of\\ concurrent systems}
\par\medskip
  \begin{minipage}[t]{.35\linewidth}
    \begin{center}
     {\Large Samy Abbes}\\[.5em]
    {\small Université Paris Cité\\
IRIF CNRS UMR 8243\\
8 place Aurélie Nemours\\
F-75013 Paris, France\\
\texttt{abbes@irif.fr}\\
ORCID: 0003-3382-3647}
\end{center}
\end{minipage}
\qquad
\begin{minipage}[t]{.35\linewidth}
  \begin{center}
{\Large Vincent Jugé}\\[.5em]
\small Université Gustave Eiffel\\
CNRS, LIGM\\
F-77454 Marne-la-Vallée, France\\
\texttt{vincent.juge@univ-eiffel.fr}\\
ORCID: 0003-0834-9082
  \end{center}
\end{minipage}
\end{center}

\bigskip
\begin{abstract}
A concurrent system is defined as a monoid action of a trace monoid on a finite set of states.
Concurrent systems represent state models where the state is distributed and where state changes are local.

Starting from a spectral property on the combinatorics of concurrent systems, we prove the existence and uniqueness of a Markov measure on the space of infinite trajectories relatively to any weight distributions.
In turn, we obtain a combinatorial result by proving that the kernel of the associated Möbius matrix has dimension~$1$;
the Möbius matrix extends in this context the Möbius polynomial of a trace monoid.

We study ergodic properties of irreducible concurrent systems and we prove a Strong law of large numbers.
It allows us to introduce the speedup as a measurement of the average amount of concurrency within infinite trajectories.
Examples are studied.

\end{abstract}

\section{Introduction}
\label{sec:introduction}

In this paper we build on previous results to study the theory of probabilistic concurrent systems.
Our main goal is to obtain ergodicity results, in particular a Strong law of large numbers with an adequate notion of convergence.
The latter point is central since ``probabilistic concurrent processes'' are not given as standard probabilistic processes.
Indeed, the very nature of concurrent systems is the absence of a global clock at the scale of the system.

We first recall some background on the theory of trace monoids, on concurrent systems and on their probabilistic counterparts.
Then we establish a spectral property which will play a central role.
We study the structure of probabilistic valuations; probabilistic valuations describe Markov measures for concurrent systems.
Then we apply these results to the study of ergodic properties, deriving in particular a Strong law of large numbers.
Examples illustrating some specific aspects of the theory are given in a separate section.

Some of the constructions and results of this paper are generalizations of statements proved in earlier works~\cite{abbes2019,abbes21}.

\paragraph{Motivations and framework.}

Markov chains are, among other things, a powerful and versatile model for describing the evolution of state systems through time.
Today, a number of important real-life systems have their state distributed over space.
As a result, the evolution of their state through time is most adequately captured by a \emph{partially ordered} collection of time instants, rather than by a \emph{sequence} of time instants.
This renders the concurrency of local changes of state, but does not comply well with Markov chains models.

Before we come to their probabilistic part, let us first quickly describe the underlying dynamics of the models that we consider.
The basic dynamics is handled by a \emph{trace monoid}, alternatively called in the literature a \emph{heap monoid} or a \emph{free partially commutative monoid}~\cite{diekert95,diekert90,cartier69,viennot86}; it also corresponds to the monoid counterpart of a \emph{right-angled Artin-Tits group}.
Its elements, called \emph{traces}, are sequences of actions modulo commutation of concurrent actions;
the concurrency feature of the model is obtained through the commutativity of some generators of the monoid.
Given a trace monoid~$\M$ and a set of states~$X$, we assume given a right monoid action~$X\times\M\to X$, which may be only partially defined.
The whole data defines a \emph{concurrent system}, a model that combines the intrinsic concurrency features of trace monoids with a natural notion of state.

\paragraph{Combinatorial encoding.}

A key point is the existence of a normal form for traces~\cite{cartier69}.
As a result, the traces identify with the paths of some finite digraph (directed graph).
This combinatorial encoding of elements is a feature shared by all monoids with general Artin-Tits presentations~\cite{dehornoy14}.
The interplay between the intrinsic properties of the concurrent system and the properties of its combinatorial encoding is one of the most interesting aspects of this theory.
We insist on that point on several occasions throughout the paper.

The probabilistic dynamics of a concurrent system is defined through a probability measure on the space of infinite traces.
Specifically, we study the notion of \emph{Markov measure}, which must satisfy an intrinsic chain rule;
``intrinsic'' meaning ``without reference to the combinatorial encoding of traces''.
The corresponding dynamics of the system is \emph{not} sequential: it is not the result of a sequence of random actions.

The mere existence of Markov measures is not obvious;
in particular, they admit natural probabilistic parameters subject to polynomial normalization conditions.

It is known that a Markov measure induces on the combinatorial encoding of traces a Markov chain dynamics, namely the \emph{Markov chain of state-and-cliques}~\cite{abbes2019}.
Yet, even for an irreducible concurrent system, the digraph that supports the representation of its trajectories may very well be non-strongly connected, and in particular it may have several final components.
This is intriguing, and it is also an obstacle to transferring in a straightforward way ergodicity results from the Markov chain of state-and-cliques (which is not ergodic in general), to the concurrent process itself (which is expected to be ergodic in some sense).

\paragraph{Contributions and outline of the paper.}

We build on previous results on probabilistic concurrent systems; hence, the characterization of a Markov measure, the existence of the Markov chain of state-and-cliques and its characteristics are recalled without proof.

After the background on trace monoids and concurrent systems has been recalled in Section~\ref{sec:trace-monoids}, we prove in Section~\ref{sec:spectr-prop-irred} a spectral property for irreducible concurrent systems (Theorem~\ref{thr:1}).
It is combined in Section~\ref{sec:constr-prob-valu} with some properties of \emph{reducible} non-negative matrices to derive information on the combinatorial encoding.
In particular, we show that the Markov chain of state-and-cliques actually lives on a sub-digraph of the encoding digraph.
Hence, some of the state-and-cliques are immaterial from the probabilistic point of view---they are called \emph{non-stable}.
Interestingly, these non-stable state-and-cliques do not depend on the particular Markov measure considered (Corollary~\ref{cor:3}).

An existence and uniqueness result on Markov measures for concurrent systems is proved (Theorem~\ref{thr:2}).
Any Markov measure is shown to be the limit of Boltzmann-like probability distributions.
With the existence and uniqueness result come a series of corollaries; in particular, the Möbius matrix of the concurrent system (which extends to concurrent systems the classical notion of Möbius polynomial) has a one-dimensional kernel generated by a positive vector (Theorem~\ref{thr:3}).

Building on the previous results, Section~\ref{sec:ergod-prop-irred} is devoted to the ergodic properties of irreducible concurrent systems.
We proceed in two steps:
we first prove that shift-invariant functions are almost surely constant (Theorem~\ref{thr:4});
then we derive a Strong law of large numbers (Theorem~\ref{thr:6}).
Again, a challenge is to formulate the ergodic properties in an intrinsic manner, without reference to the combinatorial representation of trajectories.
A key tool is a custom notion of stopping time, well suited for concurrent systems.

The Strong law of large numbers that we obtain is formulated through a smooth notion of convergence.
We consider \emph{test functions}, that is, functions which are either non-decreasing and sub-additive along finite traces, or simply additive along finite traces.
The ratios with the length of traces define the ergodic means of our test functions.
We show that, for almost surely every trajectory, the ergodic means converge toward a constant, whatever the ``way'' of convergence of finite traces toward their limit infinite trajectory.

As an application, we introduce the notion of \emph{speedup} for an irreducible concurrent system.
It  characterizes the amount of parallelism observed on average along an infinite trajectory of the given concurrent system.
Almost surely, this quantity does not depend on the infinite trajectory.
The stationary probability distribution of the Markov chain of state-and-cliques provides a way to compute the speedup (Proposition~\ref{prop:2}).

Throughout the paper we introduce small examples where all the constructions are illustrated.
In Section~\ref{sec:applications} we also introduce more elaborated examples.
The first one, in~\S~\ref{sec:an-example-with}, proves that
the combinatorial encoding of an irreducible concurrent system can have several non-isomorphic basic components.
Finally, we explicitly determine in~\S~\ref{sec:doubled-trace-monoid} the uniform measure for a family of models which directly originate from concurrency theory.

\section{Trace monoids and concurrent systems}
\label{sec:trace-monoids}

\subsection{Trace monoids (\cite{cartier69,viennot86,diekert90,dehornoy14})}
\label{sec:trace-monoids-1}

\paragraph{Algebraic description of traces.}

A \emph{trace monoid} is a presented monoid of the form:
\begin{gather*}
\M(\Sigma,I)=\langle\Sigma\,|\, \forall (a,b)\in I\quad ab=ba\rangle
\end{gather*}
where~$\Sigma$ is a finite set of \emph{letters} and~$I$ is an irreflexive and symmetric binary relation on~$\Sigma$.
Hence,~$\M=\M(\Sigma,I)$ is the quotient monoid~$\M=\Sigma^*/\I$ where~$\I$ is the smallest congruence on~$\Sigma^*$ generated by pairs of the form~$(ab,ba)$ for~$(a,b)$ ranging over~$I$.

$\M$~is \emph{trivial} if~$\Sigma=\emptyset$, \emph{non-trivial} otherwise.

Let~$R=(\Sigma\times\Sigma)\setminus I$.
The monoid~$\M$ is \emph{irreducible} if the graph~$(\Sigma,R)$ is connected;
this is equivalent to saying that~$\M$ is not isomorphic to a direct sum of the form~$\M_1\oplus\M_2$, where~$\M_1$ and~$\M_2$ are two non-trivial trace monoids.

Elements of~$\M$ are called \emph{traces}.
The unit element of~$\M$ is the \emph{empty trace}, denoted by~$\ve$.
The \emph{length} of a trace~$x$, denoted by~$\len x$, is the length of every representative word of~$x$.

The left division relation on~$\M$ is a partial order, denoted by~$\leqslant$, and defined by~$x\leqslant y$ if and only if there exists a trace~$z \in \M$ such that~$y = x z$.
Since trace monoids are left-cancellative, when this trace~$z$ exists, it is unique, and it is denoted by~$x^{-1} y$.

Trace monoids are also right-cancellative, and therefore, when~$x$ right-divides~$y$, the unique trace~$z \in \M$ such that~$y=x z$ is denoted by~$y x^{-1}$.

Furthermore, both the left and right divisibility orders satisfy the two following properties:
\begin{compactenum}[(P1)] 
\item\label{item:1} Every non-empty set of traces has a greatest lower bound.
\item\label{item:2} Every bounded, non-empty, finite set of traces has a least upper bound.
\end{compactenum}

\paragraph{Combinatorial description of traces.}

A \emph{clique} of~$\M$ is a trace of the form~$a_1\ldots a_k$, where the~$a_i$s are letters such that~$(a_i,a_j)\in I$ for all~$i \neq j$.
The empty trace is the only clique of length~$0$, and the letters of~$\Sigma$ form the cliques of length~$1$.
There exist cliques of length greater than~$1$ if and only if~$\M$ is not a free monoid.
Let~$\C$ denote the set of cliques, and~$\Cstar$ the set of non-empty cliques, which are both finite sets.

By construction, a clique is entirely determined by the letters it carries; since a clique~$c$ is the product of its letters, regardless of their order.
It is thus legitimate to identify cliques with subsets of~$\Sigma$, and to consider set-theoretic operations on cliques.

A pair~$(c_1,c_2)\in\C\times\C$ is \emph{normal} if, for every letter~$b\in c_2$, there is a letter~$a\in c_1$ such that~$(a,b)\notin I$; this is denoted by~$c_1\to c_2$.
For every trace~$x\in\M$, there exists a unique integer~$n\geqslant 0$ and a unique sequence of non-empty cliques~$(c_1,\ldots,c_n)$ such that~$x=c_1\cdots c_n$ and~$c_i\to c_{i+1}$ for all~$i\in\{1,\ldots,n-1\}$.
The sequence~$(c_1,\ldots,c_n)$ is the \emph{Cartier-Foata normal form}, or \emph{normal form} for short, of~$x$.
The integer~$n$ is the \emph{height} of~$x$, denoted by~$n=\height x$.

If~$x,y\in\M$ are two traces such that~$x\leqslant y$, the normal form of~$x$ is not, in general, a prefix of the normal form of~$y$, and the normal form of~$x^{-1}y$ is not a suffix of the normal form of~$y$.
For instance, in~$\M=\langle a,b,c\,|\, ab=ba\rangle$,~$x=a$ and~$y=ab$ satisfy~$x\leqslant y$, but both traces have height~$1$; the normal form of~$x$, which if~$(a)$, is not a prefix of the normal form of~$y$, which is~$(ab)$.

In order to characterize the relation~$x\leqslant y$ on the normal forms of~$x$ and~$y$, it is useful to introduce this notion: two cliques~$c$ and~$d$ are \emph{parallel}, which is denoted by~$c\parallel d$, whenever, as subsets of~$\Sigma$, they satisfy:~$c\cap d=\emptyset$ and~$c\cup d\in\C$.
In this case, the monoid concatenation~$cd$ is then the clique~$c\cup d$.

Now, if~$x$ and~$y$ are two traces with normal forms~$(c_1,\ldots,c_n)$ and~$(d_1,\ldots,d_p)$, respectively, then~$x\leqslant y$ if and only if (\cite[Lemma~4.1]{abbes15}):~$n\leqslant p$, and there exist cliques~$\gamma_1,\ldots,\gamma_n$ such that, for all~$i\in\{1,\ldots,n\}$:
\begin{gather}
  \label{eq:75}
  \gamma_i\parallel c_i,\ldots,c_n\text{\quad and\quad}d_i=c_i\gamma_i
\end{gather}

As a consequence, let~$x$ be a trace, of height~$n$.
For every trace~$z$ with height~$\height z\geqslant n$ and with normal form~$(d_1,\ldots,d_p)$, let~$z_n=d_1\cdots d_n$.
Then we have:
\begin{gather}
  \label{eq:71}
  x\leqslant z\iff x\leqslant z_n
\end{gather}

\paragraph{Generalized normal form of traces and boundary at infinity of the trace monoid.}

For each trace~$x$ with normal form~$(c_1,\ldots,c_n)$, let us define~$c_i$ for~$i>n$ by putting~$c_i=\ve$.
The infinite sequence~$\seq ci[1]$ thus obtained is the \emph{generalized normal form} of~$x$.
By construction, it satisfies:
\begin{inparaenum}[1:]
  \item~$c_i\to c_{i+1}$ for all~$i\geqslant1$; and
  \item there exists an integer~$N$ such that~$c_N=\ve$ (and then, necessarily,~$c_j=\ve$ for all~$j\geqslant N$).
\end{inparaenum}

An infinite sequence of cliques~$x=\seq ci[1]$ satisfying~$c_i\to c_{i+1}$ for all~$i\geqslant1$ is called a \emph{generalized trace}.
We define:~$C_i(x)=c_i$ for all integers~$i\geqslant1$.
Consider the digraph~$(\C,\to)$, with cliques as vertices and with edge relation defined by all normal pairs of cliques.
Then each generalized trace identifies in a unique way with an infinite path in the digraph~$(\C,\to)$

The set of generalized traces is denoted by~$\Mbar$.
As a closed subset of the infinite product~$(\C\times\C\times\cdots)$,~$\Mbar$~is a compact metric space.
We identify the monoid~$\M$ with a subset of~$\Mbar$; namely, each trace~$x\in\M$ identifies with its generalized normal form.

The set~$\BM=\Mbar\setminus\M$ is the \emph{boundary at infinity} of~$\M$.
Its elements are called \emph{infinite traces}.
By contrast with traces, an infinite trace is a sequence~$\omega=\seq ci[1]$ of \emph{non-empty} cliques such that~$c_i\to c_{i+1}$ for all~$i\geqslant1$.

If~$x\in\M$ and~$u\in\Mbar$, we define the concatenation~$xu$ as follows.
Let~$u=(c_i)_{i\geqslant1}$ with~$c_i\in\C$ for all~$i\geqslant1$.
Consider the sequence~$\seq x n[1]$ defined by~$x_n=xc_1\dots c_n$.
Let~$(d_{n,i})_{i\geqslant1}$ be the generalized normal form of~$x_n$.
For each~$i\geqslant1$, the sequence~$(d_{n,i})_{n\geqslant1}$ is eventually constant, say equal to~$d_i$, and then~$d_i\to d_{i+1}$ holds for all~$i\geqslant1$.
By definition, the concatenation~$xu$ is the generalized trace~$(d_i)_{i\geqslant1}$.
This extends the monoid concatenation~$xu$ if~$u\in\M$.

Let~$\xi,\omega\in\Mbar$.
We define a relation $\leqslant$ on $\Mbar$ by:
\begin{gather}
  \label{eq:73}
  \xi\leqslant\omega\iff\bigl(\forall n\geqslant1\quad C_1(\xi)\cdots C_n(\xi)\leqslant C_1(\omega)\cdots C_n(\omega)\bigr)
\end{gather}

The following proposition gathers results that mostly belong to folklore and derive essentially from the equivalence~\eqref{eq:71}.

\begin{proposition}
  \label{prop:3}
  For every generalized trace~$\xi$ and every integer~$n\geqslant0$, let~$\xi_n$ denote the trace of height~$n$ defined by~$\xi_n=C_1(\xi)\cdots C_n(\xi)$.
  \begin{compactenum}
  \item\label{item:qwpoij} The relation~$\leqslant$ on~$\Mbar$ defined by~\eqref{eq:73} is a partial ordering that extends the left-divisibility order~$(\M,\leqslant)$.
  \item\label{item:qwwqwq9} If~$u\in\M$ is of height~$n$ and if~$\xi\in\Mbar$, then the following statements are equivalent:
    \begin{inparaenum}[(i)]
    \item~$u\leqslant\xi$;
    \item~$u\leqslant\xi_n$;
    \item there exists~$\zeta\in\Mbar$ such that~$\xi=u\zeta$.
    \end{inparaenum}
  \item\label{item:4987} Every bounded subset~$Y$ of\/~$(\Mbar,\leqslant)$ has a least upper bound in\/~$(\Mbar,\leqslant)$, denoted by~$\bigvee Y$.
  \item For every~$\xi\in\Mbar$:\quad~$\xi=\bigvee\{x\in\M\tq x\leqslant\xi\}=\bigvee\{\xi_n\tq n\geqslant1\}$.
  \end{compactenum}
\end{proposition}

If a finite trace~$u$ and~$\xi\in\Mbar$ are such that~$u\leqslant\xi$, then the generalized trace~$\zeta$ such that~$\xi=u\zeta$ stated in point~\ref{item:qwwqwq9} above is unique; it is denoted by~$\zeta=u^{-1}\xi$.
In general, the cliques of~$\zeta$ are not obtained by a translation from the cliques of~$\xi$.

\paragraph{Combinatorics of trace monoids;
valuations.}

A \emph{valuation} on~$\M$ is a function~$\lambda:\M\to\RR_{>0}$ such that~$\lambda(x y)=\lambda(x)\lambda(y)$ for all~$x,y\in\M$.
Valuations are in bijection with finite families of positive numbers of the form~$(\lambda(a))_{a\in\Sigma}$.

A \emph{uniform valuation} is a valuation~$\lambda$ such that~$\lambda(a)$ is constant, say equal to~$t$, for~$a$ ranging over~$\Sigma$;
hence,~$\lambda(x)=t^{\len x}$.
The \emph{counting valuation} corresponds to~$t=1$.

Let~$G_\lambda(z)$ denote the generating series:
\begin{gather*}
G_\lambda(z)=\sum_{x\in\M}\lambda(x)z^{\len x}
\end{gather*}
and let~$\rho_\lambda$ be its radius of convergence.

The series~$G_\lambda(z)$, with non-negative coefficients, is actually a rational series;
its inverse is the \emph{$\lambda$-Möbius polynomial}:
\begin{gather}
\label{eq:1}
\mu_\lambda(z)=\sum_{\gamma\in\C}\lambda(\gamma)(-1)^{\len \gamma}z^{\len \gamma}
\end{gather}

If~$\M$ is non-trivial, the polynomial~\eqref{eq:1} has a unique complex root of smallest modulus;
this root is positive and lies in~$(0,1]$, and coincides with~$\rho$.
Furthermore, this root is simple if~$\M$ is irreducible~\cite{goldwurm00,krob03}.

Without further precision, the \emph{Möbius polynomial}~$\mu(z)$ of~$\M$ corresponds to the counting valuation;
it is thus given by~$\mu(z)=\sum_{\gamma\in\C}(-1)^{\len \gamma}z^{\len\gamma}$.

\paragraph{Möbius transform.}

The \emph{Möbius transform} of a function~$f:\C\to\RR$ is the unique function~$h:\C\to\RR$ such that:
\begin{gather}
\label{eq:2}
\forall\gamma\in\C\quad f(\gamma)=\sum_{\gamma'\in\C\tq\gamma\leqslant\gamma'}h(\gamma')
\end{gather}
It is given by:
\begin{gather}
\label{eq:3}
\forall\gamma\in\C\quad h(\gamma)=\sum_{\gamma'\in\C\tq\gamma\leqslant\gamma'}(-1)^{\len{\gamma'}-\len\gamma}f(\gamma')
\end{gather}

This is a particular instance of the general notion of Möbius transform introduced by Rota~\cite{rota64}.

\begin{remark}[connecting Möbius transform and Möbius polynomial]
Let~$\lambda$ be a valuation, and let~$f_z(x)=\lambda(x)z^{\len x}$, which is also a valuation.
Let~$h_z:\C\to\RR$ be the Möbius transform of~$f_z$.
Then the~$\lambda$-Möbius polynomial~\eqref{eq:1} coincides with~$\mu_\lambda(z)=h_z(\ve)$.
\end{remark}

\paragraph{Normal and visual cylinders.}

Let~$\Fbar$ be the Borel \slgb\ on~$\Mbar$.
By definition,~$\Fbar$~is generated by the \emph{normal cylinders}, of the following form:
\begin{gather}
\label{eq:32}
\Cbar_{x}= \{\xi\in\Mbar\tq C_i(\xi)=C_i(x),\ 1\leqslant i\leqslant \height x\}\quad(x\in\M)
\end{gather}

But~$\Fbar$ is also generated by the countable collections of \emph{full visual cylinders}:
\begin{gather*}
\Up x=\{\xi\in\Mbar\tq x\leqslant\xi\}\quad(x\in\M).
\end{gather*}

Similarly, the set~$\BM$ is equipped with its Borel \slgb~$\F$, which has two natural families of generators.
On the one hand, we have the collection of \emph{normal cylinders}:
\begin{gather}
\label{eq:6}
  C_{x}=\{\omega\in\BM\tq C_i(\omega)=C_i(x),\ 1\leqslant i\leqslant \height x\}\qquad(x\in\M)
\end{gather}
and on the other hand,  we have the collection of \emph{visual cylinders}:

\begin{gather*}
\up x=\{\omega\in\BM\tq x\leqslant\omega\}\qquad(x\in\M).
\end{gather*}

Let us justify that~$\F=\sigma\langle\up x,\ x\in\M\rangle$; the argument is similar to see that~$\Fbar=\sigma\langle\Up x,\ x\in\M\rangle$.
Let~$x\in\M$, say of height~$n=\height x$, let~$\omega\in\BM$ and let~$z=C_1(\omega)\ldots C_n(\omega)$.
Then, by Proposition~\ref{prop:3}, point~\ref{item:qwwqwq9}:~$x\leqslant\omega$ if and only if~$x\leqslant z$, and therefore:

\begin{align}
  \label{eq:64}
  \up x&=\bigsqcup_{\substack{z\in\M\tq\height z=\height x\\\ x\leqslant z}}C_z
  &C_x=\,\up x\setminus\Bigl(\bigcup_{\substack{z\in\M\tq \height z=\height x\\ z\leqslant x,\ z\neq x}}\up z
  \Bigr)
\end{align}

\paragraph{Memoryless measures and probabilistic valuations.}

Since the collection~$\{\up x\tq x\in\M\}\cup\{\emptyset\}$ is closed under finite intersections (this is called a~$\pi$-system) and generates the \slgb~$\F$, any probability measure~$m$ on~$\BM$ is entirely determined by the function~$\lambda:\M\to\RR_{\geqslant0},\quad x\mapsto m(\up x)$.

\begin{definition}
  \label{def:3}
  A probability measure~$m$ on~$\BM$ is a \emph{memoryless measure} if the function~$\lambda:\M\to\RR_{\geqslant0},\quad x\mapsto m(\up x)$ is a valuation.
Such a valuation~$\lambda$ is said to be \emph{probabilistic}.
\end{definition}

By definition, memoryless measures are thus the probability measures on~$\BM$ that satisfy~$m\bigl(\up(xy)\bigr)=m(\up x)m(\up y)$ for all~$x,y\in\M$.
Hence, this notion extends to trace monoids the standard notion of memoryless measures on infinite words.
It is a sort of ``coin tossing with partial commutations''.

Let~$\lambda$ be a valuation on~$\M$, and let~$h:\C\to\RR$ be the Möbius transform of~$\lambda$ (more precisely, of the restriction of~$\lambda$ to~$\C$).
Then~$\lambda$ is probabilistic if and only if (\cite[Th.~2.4]{abbes15}):
\begin{gather}
\label{eq:12}
h(\ve)=0,\qquad\text{and\quad}\forall\gamma\in\Cstar\quad h(\gamma)\geqslant0
\end{gather}

In particular, there is a unique valuation which is both uniform and probabilistic.
It is given by~$\lambda(x)=\rho^{\len x}$, where~$\rho$ is the root of smallest modulus of the Möbius polynomial~$\mu(z)=\sum_{\gamma\in\C}(-1)^{\len \gamma}z^{\len\gamma}$.
The associated memoryless measure is the \emph{uniform measure} on~$\BM$.

\paragraph{Markov chain of cliques.}

Let~$m$ be a memoryless measure on~$\BM$.
Then the sequence of cliques~$(C_i(\omega))_{i\geqslant1}$ of a random infinite trace~$\omega$ is a random process with values in the finite set~$\Cstar$.
It happens to have a rather simple structure.
Indeed (\cite[Th.~2.5]{abbes15}), the sequence~$\seq Ci[1]$ is a Markov chain, with initial distribution and transition matrix on~$\Cstar$ given as follows;
let~$\lambda$ be the probabilistic valuation associated to~$m$, and let~$h:\C\to\RR$ be the Möbius transform of~$\lambda$.
Furthermore, let~$g:\Cstar\to\RR_{>0}$ be the function defined by:
\begin{gather}
\label{eq:4}
\forall\gamma\in\Cstar\quad g(\gamma)=\sum_{\gamma'\in\Cstar\tq \gamma\to\gamma'}h(\gamma)
\end{gather}
Then the initial distribution of the chain is given by~$h$, which is indeed a probability distribution on~$\Cstar$; and
the transition matrix of the chain, say~$P$, is given by:
\begin{gather}
\label{eq:5}
\forall(\gamma,\gamma')\in\Cstar^2\quad P_{\gamma,\gamma'}=\un{\gamma\to\gamma'}\frac{h(\gamma')}{g(\gamma)}
\end{gather}

\begin{definition}
  \label{def:4}
  Let~$m$ be a memoryless probability measure on~$\BM$.
The \emph{Markov chain of cliques} is the probabilistic process~$(C_i)_{i\geqslant1}$ defined on the probability space~$(\BM,m)$, and with values in the finite set\/~$\Cstar$ of non-empty cliques.
\end{definition}

The Markov chain of cliques is not stationary in general.
It is aperiodic and irreducible, hence ergodic, when the associated monoid is irreducible.

\paragraph{Examples.}

The constant valuation~$f=1$ is probabilistic if and only if the monoid is commutative.
If~$\M=\Sigma^*$ is a free monoid, then~$\BM$ is the space of infinite sequences of letters;
memoryless measures coincide with memoryless measures in the usual sense on infinite sequences, and the Markov chain of cliques is actually an \iid\ sequence.

Let~$\M=\langle a,b,c,d\,|\,ac=ca,\ ad=da,\ bd=db\rangle$.
The Möbius polynomial is~$\mu(t)=1-4t+3t^2$ and~$\rho=1/3$.
Hence, the uniform measure on~$\BM$ gives equal weight~$1/3$ to the \emph{four} letters.
The corresponding valuation~$f(x)=(1/3)^{\len x}$ has the following Möbius transform~$h$, which gives the initial distribution of the Markov chain of cliques on~$\Cstar$:
\begin{gather*}
\begin{array}{cccccccc}
\text{clique~$\gamma$}&a&b&c&d&ac&ad&bd\\
f(\gamma)&1/3&1/3&1/3&1/3&1/9&1/9&1/9\\
h(\gamma)&1/9&2/9&2/9&1/9&1/9&1/9&1/9
\end{array}
\end{gather*}
The digraph of cliques is depicted on Fig.~\ref{fig:pojqwdoijasx}.

\begin{figure}
$$
\xymatrix{
*+[F]{\strut a}\POS!L!U(.3)\ar@(ul,u)[]!U\POS[]\ar@{<->}[rr]&&
*+[F]{\strut b}\ar@{<->}[rr]\POS!U(.3)!L\ar@(ul,u)[]!U&& 
*+[F]{\strut c}\ar@{<->}[rr]\POS!U(.3)!R\ar@(ur,u)[]!U&&
*+[F]{\strut d}\POS!R!U(.3)\ar@(ur,u)[]!U\\
&&*+[F]{\strut ac}
\POS!L\ar@(l,d)[]!D\POS[]
\POS!U!L\ar[ull]!D!R\POS[]\ar@{<->}[u]\POS!U!R\ar[urr]\POS[]\POS!U(.25)\ar[urrrr]!D(.25)\POS[]\ar@{<->}[rr]
&&*+[F]{\strut bd}
\POS!R\ar@(r,d)[]!D\POS[]
\POS!U(.25)\ar[ullll]!D(.25)\POS!U!L\ar[ull]\POS[]\ar@{<->}[u]\POS!U!R\ar[urr]!D!L\\
&&&*+[F]{\strut ad}
\POS!L\ar@/^1.3em/[uulll]!D
\POS!R\ar@/_1.3em/[uurrr]!D
\POS[]\ar@{<->}[ul]\ar@/_/[uul]!D!R
\POS[]\ar@{<->}[ur]\ar@/^/[uur]!D!L
\POS[]\POS!L!D(.3)\ar@(l,d)[]!D
}
$$
\caption{Digraph of cliques for the dimer model on four generators}
\label{fig:pojqwdoijasx}
\end{figure}
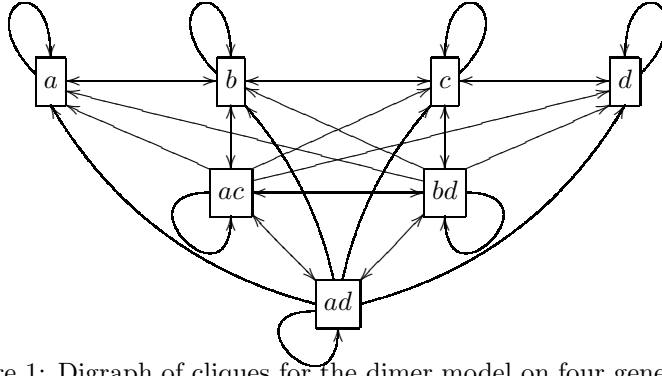

\subsection{Concurrent systems}
\label{sec:concurrent-systems}

\paragraph{Concurrent systems and their trajectories.}

We define a \emph{concurrent system} as a pair~$\SS=(\M,X)$ where~$\M$ is a trace monoid and~$X$ is a finite set of \emph{states}, together with a right monoid action~$X\sqcup\{\bot\}\times \M\to X\sqcup\{\bot\}$ where~$\bot$ is a special symbol not in~$X$, which is a sink state for the action.
Hence, we require~$\alpha\cdot\ve=\alpha$,~$\alpha\cdot(x y)=(\alpha\cdot x)\cdot y$ and~$\bot\cdot x=\bot$.

We say that~$\SS$ is \emph{trivial} if~$\alpha\cdot x=\bot$ for all~$\alpha\in X$ and~$x\in \M$;
and that~$\SS$ is \emph{non-trivial} otherwise.

We introduce the following notations, for~$(\alpha,\beta)\in X\times X$~:
\begin{align*}
\M_\alpha&=\{x\in \M\tq \alpha\cdot x\neq\bot\},
& 
                                                   \C_\alpha&=\C\cap\M_\alpha, \\
\M_{\alpha,\beta}&=\{x\in\M\tq\alpha\cdot x=\beta\},&
\C_{\alpha,\beta}&=\C\cap\M_{\alpha,\beta},
\\
\BM_\alpha&=\rlap{$\{\omega\in\BM\tq \text{for all~$x \in \M$, if~$x\leqslant\omega$, then~$x\in\M_\alpha$}\}$}
\end{align*}

A pair~$(\alpha,x)$ such that~$\alpha\in X$ and~$x\in\M_\alpha$ is called a \emph{trajectory} of~$\SS$, and a pair~$(\alpha,\omega)$ such that~$\alpha\in X$ and~$\omega\in\BM_\alpha$ is called an \emph{infinite trajectory}.

The concurrent system~$\SS$ is \emph{transitive} if~$\M_{\alpha,\beta}\neq\emptyset$ for every pair~$(\alpha,\beta)\in X\times X$.
This is similar to the transitivity property for group actions.

\begin{definition}
  \label{def:2} A concurrent system~$\SS=(\M,X)$ is \emph{irreducible}  if the three following conditions are fulfilled:
\begin{compactenum}[1:]
\item~$\SS$~is transitive and non-trivial;
\item the monoid~$\M$ is irreducible;
and
\item for every state~$\alpha$ and every letter~$a\in\Sigma$, there exists a trace~$u\in\M$ such that~$\alpha\cdot(u a)\neq\bot$.
\end{compactenum}
\end{definition}

The two first conditions are natural requirements for a notion of irreducibility;
the last condition is a \emph{liveness} property, since it requires that every letter~$a$ can eventually be played starting from any state~$\alpha$.
It is the spectral property stated below (Theorem~\ref{thr:1}) that fully justifies that this definition is relevant.

\paragraph{Example.}

A toy example of irreducible concurrent system is given as follows:
$\M=\langle a,b,c\,|\,ab=ba\rangle$,~$X=\{0,1,2\}$.
The action of~$\M$ on~$X$ is depicted on Figure~\ref{fig:qwqwdqaazaxcpow};
the absence of an arrow labeled by a letter corresponds to a forbidden action, hence~$0\cdot c=\bot$ and~$1\cdot c=\bot$;
we check that the depicted action of~$\Sigma^*$ on~$X$ induces indeed an action of~$\M$ on~$X$ since it satisfies~$(\alpha\cdot a)\cdot b=(\alpha\cdot b)\cdot a$ for all~$\alpha=0,1,2$.

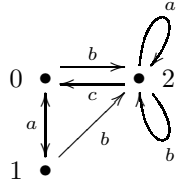
\begin{figure}
$$
\xymatrix{
\bullet\ar@<1ex>[r]^b\POS!L(2)\drop{0}&\bullet\ar@<.5ex>[l]^c\ar@(u,ur)[]^a\ar@(dr,d)[]^b\POS!R(2)\drop{2}\\
\bullet\ar@{<->}[u]^a\ar[ur]_b\POS!L(2)\drop{1}
}
$$
\caption{A concurrent system with~$X=\{0,1,2\}$ and~$\M=\langle a,b,c\,|\,ab=ba\rangle$}
\label{fig:qwqwdqaazaxcpow}
\end{figure}

\paragraph{Combinatorial encoding of trajectories.}

In order to give a faithful combinatorial encoding of trajectories, it is necessary to take into account not only the combinatorics of \emph{traces}, but also of the monoid action attached to the concurrent system.

\begin{definition}
A \emph{state-and-clique} is a pair~$(\alpha,c)\in X\times\Cstar$ such that~$\alpha\cdot c\neq\bot$.
A pair of state-and-cliques~$\bigl((\alpha,c),(\beta,d)\bigr)$ is \emph{normal}, denoted by~$(\alpha,c)\to(\beta,d)$, if~$\beta=\alpha\cdot c$ and~$c \to d$.

The digraph whose vertices are the state-and-cliques and whose edge relation coincides with~$\to$ is the \emph{directed graph of state-and-cliques}, or~\DSC.
\end{definition}

Infinite trajectories are in bijection with infinite sequences~$((\alpha_i,c_{i+1}))_{i\geqslant0}$ such that~$(\alpha_{i-1},c_i)\to(\alpha_i,c_{i+1})$ for all~$i\geqslant1$, \ie, with infinite paths in \DSC.
If~$\omega=(c_i)_{i\geqslant1}$ belongs to~$\BM_\alpha$, the corresponding infinite path in~$\DSC$ is~$(\alpha_i,c_{i+1})_{i\geqslant0}$, where~$\alpha_0=\alpha$ and~$\alpha_{i+1}=\alpha_i\cdot c_{i+1}$ for all~$i\geqslant0$.

\paragraph{Example.}

Building on the previous example of concurrent system depicted on Figure~\ref{fig:qwqwdqaazaxcpow}, the \DSC\ is represented on Figure~\ref{fig:qwdqqwfgkojngfffff}.

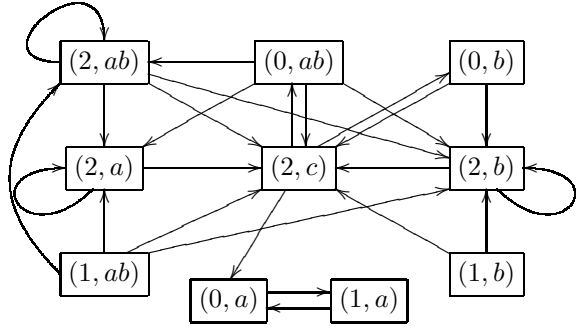
\begin{figure}
  \newsavebox{\tmpm}\savebox{\tmpm}{\xymatrix{*+[F]{(0,a)}\ar@<.7ex>[r]&*+[F]{(1,a)}\ar@<.7ex>[l]}}
  \begin{gather*}
    \xymatrix@C=4em{
      *+[F]{(2,ab)}\POS!L\ar@(l,u)[]!U\POS[]\ar[d]\POS!D!R\ar[dr]!U!L\POS[]\ar[drr]&
      *+[F]{(0,ab)}\ar[l]\POS!D!L\ar[dl]!U!R\POS[]\ar@<.6ex>[d]\ar@<-.6ex>@{<-}[d]\POS!D!R\ar[dr]!U!L&
      *+[F]{(0,b)}\ar[d]\POS!D!L\ar[dl]!U!R\POS[]\POS!L!D(.5)\ar@{<-}[dl]!U!R(.5)\\
      *+[F]{(2,a)}\ar[r]\POS!D!L(.3)\ar@(dl,l)[]!L&
      *+[F]{(2,c)}&
      *+[F]{(2,b)}\ar[l]\POS!D!R(.3)\ar@(dr,r)[]!R\\
      *+[F]{(1,ab)}\ar[u]\POS!U!R(.5)\ar[ur]!D!L\POS!U!R\ar[urr]!D!L\POS[]\POS!L\ar@(ul,dl)[uu]!D!L&
      \save[]+<0em,-1em>*{\usebox{\tmpm}}\POS!L(.75)\ar@{<-}[u]
      \restore&
      *+[F]{(1,b)}\ar[u]\ar[ul]!D!R
}
  \end{gather*}
  \caption{Directed graph of state-and-cliques (\DSC) for the previous example}
  \label{fig:qwdqqwfgkojngfffff}
\end{figure}

\paragraph{Combinatorics of concurrent systems.}
We extend the notion of valuation to concurrent systems as follows.

\begin{definition}
  \label{def:5}
A \emph{valuation} of a concurrent system~$(\M,X)$ is a collection~$\lambda=(\lambda_\alpha)_{\alpha\in X}$ of functions~$\lambda_\alpha:\M\to\RR_{\geqslant0}$ such that:
\begin{gather}
\label{eq:7}
\forall(\alpha,x)\in X\times \M\quad \lambda_\alpha(x)>0\iff \alpha\cdot x\neq\bot\\
\forall (\alpha,x,y)\in X\times \M\times\M\quad \lambda_\alpha(x y)=\lambda_\alpha(x)\lambda_{\alpha\cdot x}(y)
\end{gather}

The \emph{counting valuation} is defined by~$\lambda_\alpha(x)=\un{\alpha\cdot x\neq\bot}$.
\end{definition}

A valuation is entirely determined by the finite collection of values~$(\lambda_\alpha(a))_{(\alpha,a)\in X\times\Sigma}$.

To each valuation~$\lambda$ we attach a matrix of generating series~$G$, and a \emph{Möbius matrix}~$\Mob$, which is a matrix of polynomials;
both are square matrices of size~$X\times X$, defined by:
\begin{align*}
G_{\alpha,\beta}(z)&=\sum_{x\in\M_{\alpha,\beta}}\lambda_\alpha(x)z^{\len x}&
\Mob_{\alpha,\beta}(z)&=\sum_{c\in\C_{\alpha,\beta}}\lambda_\alpha(c)(-1)^{\len c}z^{\len c}.
\end{align*}

Then~$\Mob$ is the formal inverse of~$G$ (\cite[Th.~5.6]{abbes2019}), hence all the entries of~$G$ are rational series.

Let~$\rho_{\alpha,\beta}$ denote the radius of convergence of the~$(\alpha,\beta)$ entry of~$G$.
The \emph{characteristic root} of~$\SS$ (together with the valuation~$\lambda$) is defined as:
\begin{gather}
\label{eq:13}
\rho=\min_{(\alpha,\beta)\in X\times X}\rho_{\alpha,\beta}
\end{gather}
\begin{remark}
If the system is transitive, then~$\rho_{\alpha,\beta}$ is actually independent of~$(\alpha,\beta)$, hence~$\rho=\rho_{\alpha,\beta}$ for all pairs~$(\alpha,\beta)\in X\times X$.
If, furthermore, the system is non-trivial, then~$\rho<\infty$.
\end{remark}

\paragraph{Markov measures and probabilistic valuations.}
\label{sec:conc-syst-their}

In the following definition, we extend to concurrent systems the notions of memoryless measure and of probabilistic valuation, defined previously for trace monoids.

\begin{definition}
  \label{def:6}
A \emph{Markov measure} on a concurrent system~$(\M,X)$ is a collection~$m=(m_\alpha)_{\alpha\in X}$, where each~$m_\alpha$ is a probability measure on~$\BM$, and such that:
\begin{gather}
\label{eq:8}
\forall\alpha\in X\quad m_\alpha(\BM_\alpha)=1\\
\label{eq:50}
\forall(\alpha,x)\in X\times\M\quad x\in\M_\alpha\iff m_\alpha(\up x)>0\\
\label{eq:10}
\forall(\alpha,x,y)\in X\times\M\times\M\quad m_\alpha\bigl(\up (x y)\bigr)=m_\alpha(\up x) m_{\alpha\cdot x}(\up y).
\end{gather}

A valuation~$\lambda=(\lambda_\alpha)_{\alpha\in X}$ is a \emph{probabilistic valuation} if there exists a Markov measure~$m=(m_\alpha)_{\alpha\in X}$ such that~$\lambda_\alpha(x)=m_\alpha(\up x)$ for all~$(\alpha,x)\in X\times\M$.
\end{definition}

The conditions~\eqref{eq:50} and~\eqref{eq:10} are equivalent to saying that~$\lambda_\alpha(x)=m_\alpha(\up x)$ defines a valuation on~$\SS$.

The values~$m_\alpha(\up a)$, for~$(\alpha,a)$ ranging over~$ X\times \Sigma$, entirely determine the corresponding valuation.
These finitely many values are the natural probabilistic parameters of~$m$.
Their normalization condition is given in~\eqref{eq:11} below.

Let~$\lambda=(\lambda_\alpha)_{\alpha\in X}$ be a valuation and let~$h=(h_\alpha)_{\alpha\in X}$ be its Möbius transform;
namely, each~$h_\alpha$ is the Möbius transform (see~\S~\ref{sec:trace-monoids-1}) of the restriction~$\lambda_\alpha:\C\to\RR_{\geqslant0}$.
Then~$\lambda$ is a probabilistic valuation if and only if, for all~$\alpha\in X$~\cite[Th.~4.8]{abbes2019}:
\begin{align}
\label{eq:11}
h_\alpha(\ve)&=0,&\text{and\quad}\forall c\in\Cstar_\alpha\quad h_\alpha(c)&\geqslant0
\end{align}

\paragraph{Examples.}

Let~$\M$ act on a singleton set~$X=\{*\}$ by~$*\cdot x=*$ and~$\bot\cdot x=\bot$ for all~$x\in\M$.
Then the probabilistic valuations for the concurrent system~$(\M,X)$ correspond to the probabilistic valuations, in the sense of~\S~\ref{sec:trace-monoids-1}, of~$\M$.

Markov chains can be realized as concurrent systems with respect to a free monoid~$\M=\Sigma^*$.
Indeed, if~$E$ is the state space of a Markov chain, the corresponding concurrent systems is~$\SS=(\M,X)$ with~$X=E$,~$\M=E^*$ and~$\alpha\cdot (x_1\ldots x_n)=x_n$.
Markov measures on~$\SS$ correspond to the probability measures on the standard sample space associated with the Markov chain.

To continue the analysis of the toy example of concurrent system previously introduced and depicted on Figure~\ref{fig:qwqwdqaazaxcpow}, the following is checked to define a probabilistic valuation~$\lambda$; that is to say, to satisfy~\eqref{eq:11}:
\begin{gather}
\label{eq:36}\left\{ \begin{aligned}
\lambda_0(a)&=0.5
&\lambda_0(b)&=1
&\lambda_0(c)&=0\\
\lambda_1(a)&=0.5
&\lambda_1(b)&=1
&\lambda_1(c)&=0\\
\lambda_2(a)&=0.5
&\lambda_2(b)&=0.5
&\lambda_2(c)&=0.25
\end{aligned}\right.
\end{gather}
Indeed, if~$h$ is the Möbius transform of~$\lambda$, one has for instance:
\begin{gather*}
h_0(\ve)=1-\lambda_0(a)-\lambda_0(b)-\lambda_0(c)+\lambda_0(a)\lambda_1(b)=0.
\end{gather*}

\paragraph{The Markov chain of state-and-cliques.}

For trace monoids, we have seen that a memoryless measure gives rise to a Markov chain on cliques.

The analogous for a concurrent system equipped with a Markov measure~$\mu=(\mu_\alpha)_{\alpha\in X}$ is a Markov chain on the \DSC, relatively to every probability measure~$\mu_\alpha$.
Hence, if~$\alpha\in X$ is the initial state, the random sequence~$(\alpha_i,C_{i+1})_{i\geqslant0}$ of state-and-cliques corresponding to trajectories of~$\BM_\alpha$ is a Markov chain, characterized as follows~\cite[Th.~4.5]{abbes2019}.

Let~$f$ be the valuation associated to~$\mu$, let~$h=(h_\alpha)_{\alpha\in X}$ be the Möbius transform of~$f$, and for each~$\beta\in X$ let~$g_\beta:\C\to\RR_{\geqslant0}$ be defined by
\begin{gather}
  \label{eq:61}
g_\beta(c)=\sum_{d\in\C_{\beta}\tq c\to d}h_{\beta}(d)
\end{gather}
Then, relatively to~$\mu_\alpha$, the initial distribution of the Markov chain of state-and-cliques is~$\delta_{\{\alpha\}}\otimes h_\alpha$, and the transition matrix~$P$ is:
\begin{gather}
\label{eq:34}
P_{(\alpha,c),(\beta,d)}=\un{\beta=\alpha\cdot c}\un{c\to d}\frac{h_{\beta}(d)}{g_{\alpha\cdot c}(c)}
\qquad\text{if~$g_{\alpha\cdot c}(c)>0$},
\end{gather}
whereas~$P_{(\alpha,c),(\beta,d)}$ is undefined if~$g_{\alpha\cdot c}(c)=0$ (see below).

Furthermore, for any valuation~$f$ with Möbius transform~$h$, if one considers the function~$g$ defined by~\eqref{eq:61} and if~$h_\alpha(\ve)=0$ for all~$\alpha\in X$ (in particular, this holds if~$f$ is a probabilistic valuation), then the three functions~$f$,~$g$ and~$h$ are related by the following identity~\cite[Lemma~4.7]{abbes2019}:
\begin{gather}
\label{eq:27}
  \forall\alpha\in X\quad \forall c\in \C\quad h_\alpha(c)=f_\alpha(c)\,g_{\alpha\cdot c}(c)
\end{gather}

This has two consequences regarding the transition matrix~$P$.
 Firstly,~$g_{\alpha\cdot c}(c)=0\iff h_\alpha(c)=0$, and thus if the~$(\alpha,c)$ line of~$P$ is not defined, it corresponds to an~$(\alpha,c)$ column of~$P$ filled with~$0$s and to an initial state with~$0$ probability; hence,~$P$~is well defined by~\eqref{eq:34} as the transition matrix of a Markov chain.
Secondly, an alternative form of~$P$ is given by:
\begin{gather}
  \label{eq:62}
  P_{(\alpha,c),(\beta,d)}=\un{\beta=\alpha\cdot c}\un{c\to d}f_\alpha(c)\frac{h_\beta(d)}{h_\alpha(c)}
\end{gather}

\begin{definition}
  \label{def:7}
  Let~$(\M,X)$ be a concurrent system equipped with a Markov measure~$m$.
For each initial state~$\alpha\in X$, the associated \emph{Markov chain of state-and-cliques} is the probabilistic process~$(\alpha_i,C_{i+1})_{i\geqslant0}$ defined on the probability space~$(\BM_\alpha,m_\alpha)$.
It takes its values in the \DSC, its initial distribution is~$\delta_{\{\alpha\}}\otimes h_\alpha$, and its transition matrix is given by~\eqref{eq:34}, or equivalently by~\eqref{eq:62}.
\end{definition}

The next results in the theory of probabilistic valuations now rest upon a spectral property that we need to establish before further developments.

\section{The spectral property for irreducible concurrent systems}
\label{sec:spectr-prop-irred}

Let~$\M=\M(\Sigma,I)$ be a trace monoid.
For every letter~$a\in\Sigma$, let~$\Sigma^a=\Sigma\setminus\{a\}$ and~$I^a=I\cap(\Sigma^a\times\Sigma^a)$, and let~$\M^a$ be the submonoid of~$\M$ generated by~$\Sigma^a$, which is isomorphic to~$\M(\Sigma^a,I^a)$.

If~$\SS=(\M,X)$ is a concurrent system with~$\M=\M(\Sigma,I)$, equipped with a valuation~$\lambda$, and if~$a$ is a letter of~$\Sigma$, we define~$\SS^a=(\M^a,X)$ as the restriction of the monoid action to~$\M^a$.
It is equipped with the restriction of~$\lambda$ to~$\SS^a$.
Furthermore, let~$\rho^a$ be the characteristic root of~$\SS^a$, minimum radius of convergence of the rational series~$G^a_{\alpha,\beta}$.

From the inclusions~$\M_{\alpha,\beta}^a\subseteq \M_{\alpha,\beta}$ derives the termwise inequalities~$G^a_{\alpha,\beta}(z)\leqslant G_{\alpha,\beta}(z)$, which in turn induces the inequality of radii:
$\rho\leqslant\rho^a$.
The spectral property gives conditions for this inequality to be strict, which we shall see has many applications in the probabilistic theory of concurrent systems: see Theorem~\ref{thr:2},  Theorem~\ref{thr:3} and the results from Section~\ref{sec:consequences}.

\begin{theorem}
\label{thr:1}
Let~$\SS$ be an irreducible concurrent system.
Then~$\rho<\rho^a$ for every letter~$a$ of the base alphabet of the trace monoid.
\end{theorem}

The proof uses the custom notion of link and of linking sequence.
A \emph{link} is a trace~$a_1\ldots a_k$ such that:
\begin{compactenum}[1:]
\item every letter of~$\Sigma$ is a letter~$a_i$;
and
\item no two consecutive letters~$a_i$ and~$a_{i+1}$ satisfy~$(a_i,a_{i+1})\in I$;
and
\item~$a_1=a_k$.
\end{compactenum}
A \emph{linking sequence} is a trace~$\tau$ for which there exists a link~$\ell$ such that~$\ell^2$ is a (possibly scattered) subword of~$\tau$.
Note that, since no trace commutation relation swaps letters of~$\ell$, containing~$\ell$ or a power of~$\ell$ as a subword is indeed a trace property.

\begin{lemma}
\label{lem:1}
For every linking sequence~$\tau$ and every letter~$a\in\Sigma$, the function~$\varphi:(u,v)\in\M\times\M^a\mapsto u \tau v\in \M$ is one-to-one.
\end{lemma}

\begin{proof}
For each~$u\in\M$, let~$\gamma(u)$ denote the largest clique of~$\M$ that right-divides~$u$ and let
$u_a$ denote the largest right-divisor of~$u$ that lies in~$\M^a$, which are both well-defined thanks to Property~(P\ref{item:2}) of~\S~\ref{sec:trace-monoids-1}.

We claim that:
\begin{compactenum}
\item[(\dag)]\label{item:3} 
$u_a$ is the only trace~$x\in\M^a$ that right-divides~$u$ and such that~$\gamma(ux^{-1})\in\{\ve,a\}$
\end{compactenum}
Indeed, let~$x=u_a$;
if~$\gamma(ux^{-1})$ were divisible by a letter~$b\neq a$, the trace~$b x$ would also be a right divisor of~$u$ in~$\M^a$:
a contradiction.
Conversely, let~$x\in\M^a$ be a right divisor of~$u$ such that~$\gamma(ux^{-1})\in\{\ve,a\}$.
Then~$x$ right-divides~$u_a$, and if~$x$ were a strict right divisor of~$u_a$, there would exist a letter~$b\in\Sigma^a$ such that~$b x$ right-divides~$u_a$;
hence,~$b x$ would also right-divide~$u$, and~$b$ would divide~$\gamma(ux^{-1})$:
a contradiction.
The claim~(\dag) is proved.

Now, let~$x$ be the least left divisor of~$\tau$ that contains~$\ell$ as a subword, which exists according to Property~(P\ref{item:1}) of~\S~\ref{sec:trace-monoids-1}, and let~$z=(x^{-1}\tau)_a$.
We claim that, for every trace~$u\in\M$:
\begin{gather}
\label{eq:14}
z=(u\tau)_a
\end{gather}

Let~$y=x^{-1}\tau z^{-1}$.
By construction, and since~$\tau$ contains~$\ell^2$ as a subword,~$y z$~also contains~$\ell$ as a subword, and~$y$ contains an occurrence of all letters of~$\ell$, including~$a_1$, until the rightmost occurrence of~$a$ in~$\ell$.
Furthermore, the claim~(\dag) applied to~$x^{-1}\tau$ yields~$\gamma(y)=a$.

Let~$u\in \M$.
Since~$x$ contains an occurrence of every letter of~$\Sigma$, one has~$\gamma(ux)=\gamma(x)$.
In addition, if~$\gamma(x)$ were divisible by a letter~$b\neq a_k$, the trace~$xb^{-1}$ would be a shorter left-divisor of~$\tau$ containing~$\ell$ as a subword, contradicting the definition of~$x$.
It follows that~$\gamma(x)=a_k=a_1$.

Hence, the well known property~$\gamma(w w')=\gamma(\gamma(w) w')$ for all traces~$w,w'\in\M$, yields:
\begin{gather}
\label{eq:15}
\gamma(u x y)=\gamma(\gamma(u x) y)=\gamma(\gamma(x) y)=\gamma(a_1y)
\end{gather}

But~$y$ contains an occurrence of~$a_1$, hence~\eqref{eq:15} yields:
$\gamma(u x y)=\gamma(y)=a$.
Hence,~$\gamma(u \tau z^{-1})=\gamma(u x y)=a$, and then~(\dag)~yields~\eqref{eq:14}.

Since~\eqref{eq:14} is true for every trace~$u$, it follows that~$z=\tau_a$ and therefore that~$(u \tau v)_a=\tau_a v$ for every~$v\in\M^a$.
Consequently, given a trace~$\varphi(u,v)=u \tau v$, we have~$v=\tau_a^{-1} (u \tau v)_a$ and then~$u=(u \tau v)(\tau v)^{-1}$, which implies that~$\varphi$ is one-to-one.
\end{proof}

Let~$\myrho{N}{\alpha}$ denote, for every subset~$N\subseteq\M$ and every~$\alpha\in X$, the radius of convergence of the formal sum in~$t$:
\begin{gather*}
\mymu{N}{\alpha}=\sum_{x\in N}\lambda_\alpha(x)t^{\len x}
\end{gather*}
We shall prove:
\begin{gather}
\label{eq:17}
\forall a\in\Sigma\quad \forall\alpha\in X\quad\myrho{\M}{\alpha}<\myrho{\M^a}{\alpha}
\end{gather}

\begin{lemma}
\label{lem:2}
Let~$n$ be a positive integer.
For every state~$\alpha\in X$, denoting by~$Q_\alpha$ the set of traces of length divisible by~$n$ and that maps~$\alpha$ to itself, we have~$\myrho{\M}{\alpha}=\min\{\myrho{Q_\beta}{\beta}\tq\beta\in \alpha\cdot\M_\alpha\}$.
\end{lemma}

\begin{proof}
First, for each trace~$u\in\M_\alpha$, the inclusion~$uQ_{\alpha\cdot u}\subseteq\M_\alpha$ proves that~$\myrho{\M}{\alpha}\leqslant\myrho{Q_{\alpha\cdot u}}{\alpha\cdot u}$, and therefore that~$\myrho{\M}{\alpha}\leqslant\min\{\myrho{Q_\beta}{\beta}\tq\beta\in \alpha\cdot\M_\alpha\}$.

Then, we identify every trace~$u\in\M_\alpha$ with its lexicographically minimal representation as a word~$w$ in~$\Sigma^*$, and with the path~$p$ induced by the action of~$w$ on~$\alpha$.
That path admits a factorization, of minimal length, as a product
\begin{gather*}
p=p_1\cdot\ldots\cdot p_\ell
\end{gather*}
where~$|p_\ell|\leqslant n$ and each path~$p_i$, when~$i<\ell$, is either of length~$n$ or is a cycle of length divisible by~$n$.
Since~$\ell$ is minimal, any two paths~$p_i$ and~$p_j$ such that~$i<j$ start from distinct states, except if~$j=i+1$ and~$p_i$ is a cycle.
It follows that~$\ell\leqslant 2n$ and that~$p$ is a concatenation of up to~$|X|$ paths of length at most~$n$ and~$|X|$ cycles belonging to a set~$Q_\beta$, for pairwise distinct states~$\beta$.

Setting~$\Lambda=\max\{\lambda_\beta(a)\tq(\beta,a)\in X\times\Sigma\}$, we conclude that
\begin{gather*}
\mymu\M\alpha(t)\leqslant\sum_{k=0}^{n|X|}(\Lambda t)^k\prod_{\beta\in\alpha\cdot\M_\alpha}\mymu{Q_\beta}\beta(t)
\end{gather*}
and therefore that~$\myrho\M\alpha\geqslant\min\{\myrho{Q_\beta}\beta\tq\beta\in\alpha\cdot\M_\alpha\}$.
\end{proof}

\begin{proof}[Proof of Theorem~\ref{thr:1}.] We aim at proving~\eqref{eq:17}, which implies the result of the theorem.

Thanks to the irreducibility hypothesis on~$\SS$, we consider a link~$\ell=a_1\ldots a_k$ and, for each state~$\alpha\in X$, a linking sequence~$u_\alpha$.
We consider then a trace~$v_\alpha$ such that~$(\alpha\cdot u_\alpha)\cdot v_\alpha=\alpha$.
Then, let~$n$ be the least common multiple of all lengths~$\len{u_\alpha v_\alpha}$;
the trace~$\tau_\alpha=(u_\alpha v_\alpha)^{n/\len{u_\alpha v_\alpha}}$ is a linking sequence of length~$n$ that maps~$\alpha$ to itself.

Like in the previous proof, we identify every trace~$u\in\M_\alpha$ with its lexicographically minimal representation as a word~$w\in\Sigma^*$, and with the path~$p$ induced by the action of~$w$ on~$\alpha$

On the one hand, let~$\Lambda_\alpha$ be the set of traces that map~$\alpha$ to itself and that can be factored as a product of traces of length~$n$ such that the rightmost factor using the letter~$a$, if any, is a trace~$\tau_\beta$ for some state~$\beta\in X$.
On the other hand, let~$Q_\alpha^a$ be the set of all traces~$u\in\M^a$ of length divisible by~$n$ that map~$\alpha$ to itself, and let~$\beta_0\in \alpha\cdot\M_\alpha$ be a state such that the radius~$\myrho{Q_{\beta_0}^a}{\beta_0}$ is minimal.

According to Lemma~\ref{lem:1}, the map~$(u,v)\in\Lambda_{\beta_0}\times Q_{\beta_0}^a\mapsto u\tau_{\beta_0}v\in\Lambda_{\beta_0}$ is one-to-one, and thus:
\begin{gather*}
\mymu{\Lambda_{\beta_0}}{\beta_0}(t)\geqslant t^n \mymu{\Lambda_{\beta_0}}{\beta_0}(t)\times
\mymu{Q_{\beta_0}^a}{\beta_0}(t)
\end{gather*}
term-wise.
It follows that~$0\leqslant \mymu{Q_{\beta_0}^a}{\beta_0}(t)\leqslant1/t^n$ on the interval~$(0,\rho)$ where~$\rho=\myrho{\Lambda_{\beta_0}}{\beta_0}$.

Moreover,~$\M^a$ admits rational normal forms and~$X$ is finite, so that~$\mymu{Q_{\beta_0}^a}{\beta_0}$ is a rational series with non-negative-coefficients.
In particular, its radius of convergence is a pole;
and since~$\mymu{Q_{\beta_0}^a}{\beta_0}$ is bounded on~$(0,\rho)$, it follows that~$\myrho{Q_{\beta_0}^a}{\beta_0}>\rho$.

But~$\rho\geqslant\myrho\M\alpha$ and~$\myrho{Q_{\beta_0}^a}{\beta_0}=\myrho{\M^a}\alpha$ according to Lemma~\ref{lem:2}.
Hence,~$\myrho{\M^a}\alpha>\myrho\M\alpha$, which was to be proved.
\end{proof}

\section{Probabilistic valuations}
\label{sec:constr-prob-valu}

From now on, we consider a concurrent system~$\SS=(\M,X)$ equipped with a valuation~$f=(f_\alpha)_{\alpha\in X}$, and we assume that~$\SS$ is transitive and non-trivial.

\subsection{Weak convergence of Boltzmann-like distributions}
\label{sec:weak-conv-boltzm}

\begin{definition}[cocycle]
\label{def:1}
Let~$X$ be a set.
A \emph{cocycle} on~$X$ is a function~$\Delta:X\times X\to\RR_{\geqslant0}$ such that~$\Delta(\alpha,\gamma)=\Delta(\alpha,\beta)\Delta(\beta,\gamma)$ for all~$\alpha,\beta,\gamma\in X$.
If~$\Delta$ takes only positive values, we say that~$\Delta$ is a \emph{positive cocycle}.
\end{definition}

For every positive real~$r$ and for every positive cocycle~$\Delta$ on~$X$, the following defines another valuation~$\ftilde$ on~$\SS$:
\begin{gather}
\label{eq:18}
\forall(\alpha,x)\in X\times\M\quad \ftilde_\alpha(x)=r^{\len x}\Delta(\alpha,\alpha\cdot x)f_\alpha(x)
\end{gather}

We are looking for conditions on~$r$ and~$\Delta$ for the new valuation~$\ftilde$ to be probabilistic.
This will prove the existence of probabilistic valuations, and thus of Markov measures.

\medskip
For each~$\alpha\in X$, let~$G_\alpha(s)$ be the generating function
$G_\alpha(s)=\sum_{x\in\M_\alpha}f_\alpha(x)s^{\len x}$, of radius of convergence~$\rho$;
and for~$s\in[0,\rho)$ let~$\nu_{\alpha,s}$ be the Boltzmann-like probability distribution on~$\M_\alpha$ defined by
\begin{gather}
\label{eq:23}
\nu_{\alpha,s}=\frac1{G_\alpha(s)}\sum_{x\in\M_\alpha}f(x)s^{\len x}\delta_{\{x\}}
\end{gather}

\begin{lemma}
\label{lem:3}
For each~$\alpha\in X$, and as~${s\longrightarrow\rho^-}$, the following weak convergence holds:
\begin{gather*}
\nu_{\alpha,s}\xrightarrow{\ \textsf{w}\ }\mu_\alpha
\end{gather*}
where~$\mu_\alpha$ is the probability measure on~$\partial\M_\alpha$ entirely characterized by~:
\begin{gather}
\label{eq:19}
\forall x\in\M_\alpha\quad\mu_\alpha(\up x)=f_\alpha(x)\rho^{\len x}\Delta(\alpha,\alpha\cdot x)
\end{gather}
and~$\Delta:X\times X\to\RR_{>0}$ is the positive cocycle given by
\begin{gather*}
\Delta(\alpha,\beta)=\lim_{s\to\rho^-}\frac{G_\beta(s)}{G_\alpha(s)}
\end{gather*}
\end{lemma}

\begin{proof}
For each pair~$(\alpha,\beta)\in X\times X$ and for each~$s\in[0,\rho)$, define~$H_{\alpha,\beta}(s)$ by:
\begin{gather*}
H_{\alpha,\beta}(s)=\frac{G_\beta(s)}{G_\alpha(s)}
\end{gather*}
which is well defined since~$G_\alpha(s)>0$ on~$[0,\rho)$.
Let~$x\in\M_\alpha$ be such that~$\beta=\alpha\cdot x$.
Then, since~$f$ is a valuation and is non-negative:
\begin{gather}
\label{eq:16}
G_\alpha(s)\geqslant f_\alpha(x)s^{\len x}G_\beta(s)
\end{gather}
Since~$f_\alpha(x)>0$, it entails that~$H_{\alpha,\beta}$ is bounded on~$[\rho/2,\rho)$.
Since~$H_{\alpha,\beta}(s)$ is a rational function, it has thus a limit~$\Delta(\alpha,\beta)\in[0,+\infty)$ as~$s\to\rho^-$.
But~$H_{\beta,\alpha}(s)=\bigl(H_{\alpha,\beta}(s)\bigr)^{-1}$ for~$s\in[0,\rho)$ and also has a limit~$\Delta(\beta,\alpha)\in[0,+\infty)$.
This entails that~$\Delta(\alpha,\beta)>0$.
So far we have proved that the cocycle~$\Delta$ is well defined and positive.

Since~$\M$ identifies with a subset of~$\Mbar$, we see the probability measures~$\nu_{\alpha,s}$ as defined on the compact space~$\Mbar$ rather than on~$\M$.
From the definition~\eqref{eq:23}, we derive:
\begin{gather}
\label{eq:20}
\forall (\alpha,x)\in X\times \M\quad \nu_{\alpha,s}(\Up x)=s^{\len x}f_\alpha(x)H_{\alpha,\alpha\cdot x}(s)
\end{gather}
and therefore:
\begin{gather}
\label{eq:21}
\forall (\alpha,x)\in X\times \M\quad \lim_{s\to\rho^-}\nu_{\alpha,s}(\Up x)=f_\alpha(x)\rho^{\len x}\Delta(\alpha,\alpha\cdot x)
\end{gather}
The family~$\H=\{\emptyset\}\cup\{\Up x,\quad x\in\M\}$ is a~$\pi$-system generating the Borel \slgb\ of~$\Mbar$;
and all the elements of~$\H$ are subsets of empty topological boundary in~$\Mbar$.
Hence, by the compactness of~$\Mbar$ on the one hand and using the Portemanteau theorem~\cite{billingsley99} on the other hand, \eqref{eq:21}~ensures the weak convergence of~$\nu_{\alpha,s}$, as~$s\to\rho^-$, toward a probability measure~$\mu_\alpha$ on~$\Mbar$ satisfying:
\begin{gather}
\label{eq:22}
\forall (\alpha,x)\in X\times \M\quad \mu_\alpha(\Up x)=f_\alpha(x)\rho^{\len x}\Delta(\alpha,\alpha\cdot x)
\end{gather}
Every singleton~$\{x\}$ for~$x\in\M$ is of empty topological boundary in~$\Mbar$.
By the Portemanteau theorem, its measure is thus given by:
\begin{gather*}
\mu_\alpha(\{x\})=\lim_{s\to\rho^-}\nu_{\alpha,s}\bigl(\{x\}\bigr)=\lim_{s\to\rho^-}\frac{f_\alpha(x)s^{\len x}}{G_\alpha(s)}=0 
\end{gather*}
since the system is non-trivial.
Since~$\M$ is countable, we deduce first that~$\mu_\alpha$ is indeed a probability measure on~$\BM$, and second that in~\eqref{eq:22}, one may read~$\up x$ rather than~$\Up x$, which is thus~\eqref{eq:20}

So far~$\mu_\alpha$ is a probability measure on~$\BM$; it remains to see that it is concentrated on~$\BM_\alpha$.
Since~$\Mbar_\alpha$ is a closed subset of~$\Mbar$, passing to the limit in~$\nu_{\alpha,s}(\Mbar_\alpha)=1$ yields~$\mu_\alpha(\Mbar_\alpha)\geqslant1$, and thus~$\mu_\alpha(\Mbar_\alpha)=1$ and finally~$\mu_\alpha(\BM_\alpha)=1$ since~$\mu_\alpha(\M_\alpha)=0$.
\end{proof}

\begin{discussion}
Lemma~\ref{lem:3} shows in particular that, when starting from a valuation~$f$, there exists an adequate pair~$(r,\Delta)$ making the new valuation~$\ftilde_\alpha(x)=r^{\len x}\Delta(\alpha,\alpha\cdot x) f_\alpha(x)$ probabilistic.
This raises two natural questions: is the pair~$(r,\Delta)$ unique? and
is any Markov measure the weak limit of Boltzmann-like distributions? The first question is the topic of the next section;
the second question will be given a positive answer in~\S~\ref{sec:consequences}.
\end{discussion}

\subsection{Stable state-and-cliques and a uniqueness result}
\label{sec:constr-prob-valu-1}

The following result is a key step for the analysis of probabilistic valuations.

\begin{theorem}
  \label{thr:2}
Let an irreducible concurrent system~$(\M,X)$ be equipped with a valuation~$f$.
Then there is a unique pair~$(r,\Gamma)$ where~$r>0$ and\/~$\Gamma:X\times X\to\RR_{>0}$ is a positive cocycle, such that the valuation~$\ftilde$ defined by:
\begin{gather}
  \label{eq:9}
\forall\alpha\in X\quad\forall x\in\M_\alpha\quad  \ftilde_\alpha(x)=r^{\len x}\Gamma(\alpha,\alpha\cdot x)f_\alpha(x)
\end{gather}
is a probabilistic valuation;
this is the pair~$(\rho,\Delta)$ where~$\rho$ is the common radius of convergence of the series~$G_\alpha(s)$, for~$\alpha$ ranging over~$X$, and~$\Delta$ is the positive cocycle introduced in Lemma~{\normalfont\ref{lem:3}}.
\end{theorem}

Applying in particular Theorem~\ref{thr:2} with the counting valuation yields the following definition.

\begin{definition}
  For an irreducible concurrent system~$(\M,X)$, let\/~$(r,\Delta)$ be the unique pair where~$r>0$ and~$\Delta:X\times X\to\RR$ is a positive cocycle, such that the valuation~$f_\alpha(x)=r^{\len x}\Delta(\alpha,\alpha\cdot x)$, is a probabilistic valuation.
The associated Markov measure is the \emph{uniform measure} of\/~$(\M,X)$.
\end{definition}

The existence part of Theorem~\ref{thr:2} follows from Lemma~\ref{lem:3}.
The proof of uniqueness relies on two new ingredients.
The first ingredient is some properties of reducible non-negative matrices (matrices with non-negative coefficients).
Recall that a component (or access class)  of a non-negative matrix \cite{rothblum14,berman94,seneta81} is \emph{basic} if it is of maximal spectral radius;
and that it is \emph{final} if it does not have access to any other component.

\begin{definition}[umbrella matrix]
  A non-negative matrix~$F$ is an \emph{umbrella matrix} if its basic components exactly coincide with its final components.
\end{definition}

 It is known that for a non-negative matrix~$F$  to be an umbrella matrix, it is necessary and sufficient that there exists a positive right eigenvector (\cite[\S~10.3, Fact 12(b)]{rothblum14}).

 The second ingredient is a study, based on the result of Theorem~\ref{thr:1}, of state-and-cliques, in particular a characterization of those which we call \emph{stable}.

\begin{definition}[protection, stable state-and-clique, \DSCP]
  \label{def:qwoihj}
  Let\/~$(\alpha,c)$ be a state-and-clique (see~\S~\ref{sec:concurrent-systems}).
A trace~$x\in\M_\alpha$ is a \emph{protection} of\/~$(\alpha,c)$ if:
  \begin{gather*}
    \forall z\in\M_{\alpha}\quad x\leqslant z\implies C_1(z)=c
  \end{gather*}
  A state-and-clique~$(\alpha,c)$ is \emph{stable} if there exists a protection of\/~$(\alpha,c)$.

  The sub-digraph of\/ \DSC\ with stable state-and-cliques as vertices will be denoted by\/ \DSCP.
\end{definition}

\begin{example}
  Referring to the concurrent system depicted on Fig.~\ref{fig:qwqwdqaazaxcpow}, one sees that~$(0,a)$ is not a stable state-and-clique.
Indeed, if~$a\leqslant x$ and~$C_1(x)=a$, then~$x$ is necessarily of the form~$x=a^n$ for~$n\geqslant1$; but no such~$x$ can be a protection of~$(0,a)$ since~$C_1(a^nb)=ab$.
On the contrary,~$(0,b)$~and~$(0,ab)$ are both stable.
\end{example}

\begin{lemma}\newcounter{mycount}
  \label{lem:11}
For any non-trivial concurrent system~$\SS$:
  \begin{compactenum}
  \item\label{item:5} If\/~$(\beta,d)$ is stable and if\/~$(\alpha,c)\to(\beta,d)$, then~$(\alpha,c)$ is stable.
  \item\label{item:qwdpoij8} There exists at least one stable state-and-clique.
  \item\setcounter{mycount}{\theenumi}\label{item:9qwdqwdplkcs} If\/~$\SS$ is irreducible and if\/~$(\alpha,c)$ is a stable state-and-clique, there exists at least one stable state-and-clique~$(\beta,d)$ such that\/~$(\alpha,c)\to(\beta,d)$.
  \end{compactenum}
  Assume that~$\SS$ is irreducible and equipped with a probabilistic valuation~$\ftilde$, and let~$\htilde$ be the Möbius transform of~$\ftilde$.
  \begin{compactenum}
  \setcounter{enumi}{\themycount}\item\label{item:4} If\/~$(\alpha,c)$ is stable, then~$\htilde_\alpha(c)>0$.
  \item\label{item:6} Assume furthermore that\/~$\ftilde$ is the probabilistic valuation given by~$\ftilde_\alpha(x)=f_\alpha(x)\rho^{\len x}\Delta(\alpha,\alpha\cdot x)$, where~$\rho$ and\/~$\Delta$ have been introduced in Lemma~\ref{lem:3}.
If\/~$\htilde_\alpha(c)>0$, then\/~$(\alpha,c)$ is stable.
  \end{compactenum}
\end{lemma}

\begin{proof}
\ref{item:5}.\quad Let~$x$ be a protection of~$(\beta,d)$.
Then~$cx$ is a protection of~$(\alpha,c)$.

\ref{item:qwdpoij8}.\quad Let~$\alpha_0\in X$ and~$a_0\in\Sigma$ such that~$\alpha_0\cdot a_0\neq\bot$.
For any clique~$c$ maximal in~$\C_{\alpha_0}$, the state-and-clique~$(\alpha_0,c)$ is stable.

\ref{item:9qwdqwdplkcs}.\quad Let~$x$ be a protection of~$(\alpha,c)$.
Since~$\SS$ is assumed to be irreducible, let~$y\in\M_{\alpha\cdot x}$ be a trace with at least one occurrence of every letter of~$\Sigma$.
Let~$\beta=\alpha\cdot c$ and let~$d$ be the second clique in the normal form of~$xy$.
Then~$(\beta,d)$ is a stable state-and-clique, of which~$c^{-1}xy$ is a protection.

\ref{item:4}.\quad Let~$\mu=(\mu_\alpha)_{\alpha\in X}$ be the Markov measure associated with the probabilistic valuation~$\ftilde$.
Let~$x\in\M_\alpha$ be a protection of a stable state-and-clique~$(\alpha,c)$.
Then:
$\{\omega\in\BM_\alpha\tq C_1(\omega)=c\}\supseteq\,\, \up x$.
Taking the~$\mu_\alpha$ probabilities yields:
$\mu_\alpha(C_1=c)\geqslant \mu_\alpha(\up x)$ which writes as~$\htilde_\alpha(c)\geqslant \ftilde_\alpha(x)>0$.

\ref{item:6}.\quad Assuming that the system is irreducible, we prove the stated implication by contraposition.
Hence, letting~$(\alpha,c)$ be a non-stable state-and-clique, we show that~$\htilde_\alpha(c)=0$.
Consider the probability distributions~$\nu_{\alpha,s}$ introduced in~\eqref{eq:23} and their weak limits ($\mu_\alpha)_{\alpha\in X}$ introduced in Lemma~\ref{lem:3}.
We have, on the one hand:
  \begin{align}
\label{eq:24}
    \htilde_\alpha(c)&=\mu_\alpha(C_1=c)=\lim_{s\to\rho^-}\nu_{\alpha,s}\bigl(\{x\in\M_\alpha\tq C_1(x)=c\}\bigr)
  \end{align}
and on the other hand:
\begin{align}
  \notag
  \{x\in\M_\alpha\tq C_1(x)=c\}&\subseteq\bigcup_{a\in\Sigma}\M_\alpha^a\qquad\text{where~$\M^a=\bigl\langle\Sigma\setminus\{a\}\bigr\rangle$}
\end{align}                                 
Indeed, a trace~$x$ such that~$C_1(x)=c$ and containing an occurrence of all letters of~$\Sigma$ would be a protection of~$(\alpha,c)$, whereas~$(\alpha,c)$ is assumed to be non-stable.

Taking the~$\nu_{\alpha,s}$ probability of both members above, and denoting by~$G^a(s)$ the generating function associated with the system~$(\M^a,X)$, yields:
\begin{align}
\label{eq:25}
\nu_{\alpha,s}\bigl(
 \{x\in\M_\alpha\tq C_1(x)=c\}
  \bigr)&\leqslant\frac1{G_\alpha(s)}\sum_{a\in\Sigma}G^a_\alpha(s)
\end{align}
Theorem~\ref{thr:1} applied to the irreducible system~$(\M,X)$ implies that the right-hand member of~\eqref{eq:25} converges to~$0$ as~$s\to\rho^-$.
Comparing with~\eqref{eq:24} yields~$\htilde_\alpha(c)=0$.
\end{proof}

Most of our arguments will now focus on the following matrix: given a valuation~$f=(f_\alpha)_{\alpha\in X}$ and a positive real~$r>0$, let the square matrix~$F$ indexed by state-and-cliques be defined by:
\begin{gather}
  \label{eq:63}
  F_{(\alpha,c),(\beta,d)}=\un{\beta=\alpha\cdot c}\un{c\to d}r^{\len c}f_\alpha(c)
\end{gather}

\begin{lemma}
\label{lem:8}
  Let~$\SS$ be a concurrent system, let~$f$ be a valuation and let~$\ftilde_\alpha(x)=r^{\len x}\Gamma(\alpha,\alpha\cdot x)f_\alpha(x)$ where~$r>0$ and~$\Gamma:X\times X\to\RR$ is a cocycle.
Let~$\htilde$ be the Möbius transform of~$\ftilde$, and assume that~$\htilde_\alpha(\ve)=0$ for all~$\alpha\in X$ (in particular, this holds if~$\ftilde$ is a probabilistic valuation).

  Then for any state~$\alpha_0\in X$, the vector~$u$ defined by~$u_{(\alpha,c)}=\Gamma(\alpha_0,\alpha)\htilde_\alpha(c)$ is an invariant vector of the matrix~$F$.
\end{lemma}

\begin{proof}
For every~$\alpha\in X$, consider the function~$\gtilde_\alpha:\C\to\RR$ introduced earlier in~\eqref{eq:61}.

 Since~$\htilde_\alpha(\ve)=0$ for all~$\alpha\in X$, the identity~\eqref{eq:27} holds.
Using the definition of~$\gtilde$, the computation of~$Fu$ goes then as follows:
  \begin{align*}
    (Fu)_{(\beta,d)}&=r^{\len d}f_\beta(d)\Gamma(\alpha_0,\beta\cdot d)\,\gtilde_{\beta\cdot d}(d)\\
    &=\Gamma(\alpha_0,\beta)\,\ftilde_\beta(d)\,\gtilde_{\beta\cdot d}(d)&&\text{since~$\Gamma$ is a cocycle}\\
                    &=\Gamma(\alpha_0,\beta)\,\htilde_\beta(d)&&\text{by~\eqref{eq:27}}
  \end{align*}
  Hence,~$u$ is an invariant vector of~$F$.
\end{proof}

\begin{lemma}
  \label{lem:22}
  Let~$r>0$ and let\/~$\Gamma:X\times X\to\RR_{>0}$ be a positive cocycle such that the valuation~$\ftilde$ defined by\/~$\ftilde_\alpha(x)=r^{\len x}\Gamma(\alpha,\alpha\cdot x) f_\alpha(x)$ is probabilistic.

Let~$F$ be the matrix defined by~\eqref{eq:63} and let\/~$F^+$ be the restriction of~$F$ to stable state-and-cliques (hence~$F^+$ is the adjacency matrix of\/ \DSCP).
Then:
\begin{compactenum}
\item\label{item:11}~$r=\rho$, radius of convergence of the series~$G_\alpha(s)$.
\item\label{item:12} For any state~$\alpha_0\in X$, the vector~$u$ defined by:
\begin{gather*}
    u_{(\alpha,c)}=\Gamma(\alpha_0,\alpha)\htilde_\alpha(c)
  \end{gather*}
  where~$\htilde$ is the Möbius transform of~$\ftilde$, is a nonnegative invariant vector of~$F$.
\item\label{item:13} The matrix~$F$ admits the following block shape, where stable state-and-cliques are put first:
$F=\left(\begin{smallmatrix}
      F^+ &X\\0&F^0
    \end{smallmatrix}
  \right)$, and where:
  \begin{compactenum}
  \item\label{item:9}~$F^+$~is an umbrella matrix of spectral radius~$1$, for which the restriction~$u^+$ of~$u$ to stable state-and-cliques is a positive and invariant vector;
  \item\label{item:10} The spectral radius of~$F^0$ is~$<1$.
  \end{compactenum}
\end{compactenum}
\end{lemma}

\begin{proof}
\emph{We start by proving Point~\ref{item:12}.}\quad Referring to~\eqref{eq:11} which characterizes probabilistic valuations, and since~$\ftilde$ is assumed to be probabilistic, one sees first that~$u$ is non-negative; and second that~$u$ is~$F$-invariant  thanks to Lemma~\ref{lem:8}.

\emph{For proving the other points,} we first introduce, for~$s\geqslant0$, the square non-negative matrix~$F(s)$ indexed by state-and-cliques and defined by:
  \begin{gather*}
    F_{(\alpha,c),(\beta,d)}(s)=\un{\beta=\alpha\cdot c}\un{c\to d}s^{\len c}f_\alpha(c)
  \end{gather*}

  It follows from point~\ref{item:5} of Lemma~\ref{lem:11} that~$F(s)$ has the block shape~$F(s)=\left(
    \begin{smallmatrix}
      F^+(s)&X(s)\\0&F^0(s)
    \end{smallmatrix}
    \right)$
  where stable state-and-cliques are put first.
We prove the following claim:
\begin{compactenum}
  \item[(\dag)] Let~$F_1=F(\rho)$,~$F_1^+=F^+(\rho)$ and~$F_1^0=F^0(\rho)$.
The spectral radius of~$F^+_1$ is~$1$ and the spectral radius of~$F_1^0$ is~$<1$.
\end{compactenum}

Let~$\varphi_\alpha(x)=\rho^{\len x}\Delta(\alpha,\alpha\cdot x)f_\alpha(x)$ be the probabilistic valuation provided by Lemma~\ref{lem:3}, let~$\theta$ be the Möbius transform of~$\varphi$, and let~$v$ be the vector defined by~$v_{(\alpha,c)}=\theta\alpha(c)\Delta(\alpha_0,\alpha_0\cdot c)$.
The result of point~\ref{item:12} already proved applies to~$\varphi$, hence~$v$ is a non-negative invariant vector of~$F_1$.
  
Furthermore,~$\theta_\alpha(c)=0$ according to point~\ref{item:6} of Lemma~\ref{lem:11} if~$(\alpha,c)$ is not stable.
Hence,~$v=0$ on non-stable state-and-cliques, and therefore~$F_1^+v^+=v^+$ where~$v^+$ denotes the restriction of~$v$ to stable state-and-cliques.
But~$v^+>0$ (by point~\ref{item:4} of Lemma~\ref{lem:11}) and~$v^+$ is a non-empty vector (by point~\ref{item:qwdpoij8} of Lemma~\ref{lem:11}), which implies that~$F_1^+$ is an umbrella matrix of spectral radius~$1$.

To prove that the spectral radius of~$F_1^0$ is~$<1$ and complete the proof of~(\dag), it is enough to prove the convergence of the series of matrices~$S=\sum_{n\geqslant0}(F_1^0)^n$.
The matrix~$S$ is indexed by non-stable state-and-cliques, and for two non-stable state-and-cliques~$(\alpha,c)$ and~$(\beta,d)$, one has:
\begin{gather*}
  S_{(\alpha,c),(\beta,d)}=\sum_{x}\rho^{\len x}f_\alpha(x)
\end{gather*}
where~$x$ ranges over traces with first and last cliques~$(\alpha,c)$ and~$(\beta,d)$ respectively.
But such a trace cannot have an occurrence of all letters of~$\Sigma$ because~$(\alpha,c)$ is a non-stable state-and-clique.
Therefore:
 \begin{gather*}
   S_{(\alpha,c),(\beta,d)}\leqslant\sum_{a\in\Sigma}\Bigl(\sum_{x\in\M^a_\alpha}\rho^{\len x}f_\alpha(x)\Bigr)<\infty
 \end{gather*}
where the convergence of the series follows from Theorem~\ref{thr:1}.
The proof of (\dag) is complete.

\emph{We now come to the proof of Point~\ref{item:11}.}\quad For~$s<\rho$, one has:
\begin{gather*}
  \sum_{n\geqslant0}\bigr[\bigl(F(s)\bigr)^n\bigr]_{(\alpha,c),(\beta,d)}\leqslant G_\alpha(s)<\infty
\end{gather*}
Henceforth:
$s<\rho$ implies that~$I-F(s)$ is invertible.
But~$F(r)$ has a non-zero invariant vector according to point~\ref{item:12} already proved.
Hence,~$r\geqslant\rho$.

Seeking a contradiction, assume that~$r>\rho$.
We write~$x\in\DSCP$ to denote that the trace~$x$ has all its cliques in \DSCP.
We claim that for some constant~$M$:
\begin{gather}
  \label{eq:52}
  \forall n\geqslant0\quad
  \sum_{x\in\M_\alpha\cap\DSCP\tq\height x=n}\ftilde_\alpha(x)\leqslant M
\end{gather}
where~$\height x$ denotes the height of~$x$.
Indeed, with~$\nu=(\nu_\alpha)_{\alpha\in X}$ the Markov measure associated with the probabilistic valuation~$\ftilde$, the total probability law yields:
\begin{gather}
  \label{eq:55}
  \sum_{z\in\M_\alpha\tq\height z=n+1}\nu_\alpha(C_1\cdots C_{n+1}=z)=1
\end{gather}
Let~$z$ of height~$n+1$ with normal form~$z=c_1\cdots c_{n+1}$, and let~$x=c_1\cdots c_n$ and~$y=c_{n+1}$.
Then the form~\eqref{eq:62} for the transition matrix of the Markov chain of state-and-cliques yields~$\nu_\alpha(C_1\cdots C_{n+1}=z)=\ftilde_\alpha(x)\htilde_{\alpha\cdot x}(y)$.
Hence, from~\eqref{eq:55}:
\begin{gather}
  \label{eq:54}
  \sum_{x\in\M_\alpha\tq\height x=n}\ftilde_\alpha(x)\Bigl(\sum_{y\in\Cstar_{\alpha\cdot x}\tq x\to y}\htilde_{\alpha\cdot x}(y)\Bigr)=1
\end{gather}

If~$x\in\M_\alpha$ is bound to stay within \DSCP, then there exists, according to point~\ref{item:9qwdqwdplkcs} of Lemma~\ref{lem:11}, at least one stable state-and-cliques of the form~$(\alpha\cdot x,y)$, and all of these satisfy~$\htilde_{\alpha\cdot x}(y)>0$ according to point~\ref{item:4} of Lemma~\ref{lem:11};
hence~\eqref{eq:52} follows.

From~$r>\rho$ on the one hand, and from~\eqref{eq:52} on the other hand, we derive:
\begin{gather}
  \label{eq:56}
  \sum_{x\in\M_\alpha\cap\DSCP}\Bigl(\frac\rho r\Bigr)^{\len x}\ftilde_\alpha(x)<\infty\qquad\text{since~$\len x\geqslant\height x$}
\end{gather}
yielding:
\begin{gather}
  \label{eq:57}
  \sum_{x\in\M_\alpha\cap\DSCP}\rho^{\len x} f_\alpha(x)<\infty
\end{gather}
since~$\Gamma(\cdot,\cdot)$ is a positive cocycle.
But this contradicts that~$F_1^+$ has spectral radius~$1$, proved in~(\dag);
hence~$r=\rho$.

\emph{Point~\ref{item:13}.}\quad From~$r=\rho$ derives that~$F=F_1$.
The properties of~$F^+$ and of~$F^0$ stated in the lemma are thus the properties of~$F_1^+$ and of~$F_1^0$, already proved.
Let~$u=\left(
  \begin{smallmatrix}
    u^+\\u^0
  \end{smallmatrix}\right)$ be the block decomposition of~$u$ where stable state-and-cliques are put first.
We already saw that~$u^+>0$.
Furthermore,~$u^0$~is in invariant vector of~$F^0$, since~$u$ is an invariant vector of~$F$, but~$F^0$ has spectral radius~$<1$, hence~$u^0=0$.
Therefore~$u^+$ is an invariant vector of~$F^+$.
\end{proof}

\begin{proof}[Proof of Theorem~\ref{thr:2}.]
  The existence part is a direct consequence of Lemma~\ref{lem:3}, hence we focus on the uniqueness part of the statement.

Let~$(r,\Gamma)$ be a pair as in the statement, and let~$(\rho,\Delta)$ be the pair given by Lemma~\ref{lem:3}.
Let~$\ftilde$ and~$\varphi$ be two probabilistic valuations defined by:
  \begin{align*}
    \ftilde_\alpha(x)&=f_\alpha(x)r^{\len x}\Gamma(\alpha,\alpha\cdot x)
                       &\varphi_\alpha(x)&=f_\alpha(x)\rho^{\len x}\Delta(\alpha,\alpha\cdot x)
  \end{align*}
  and let~$\htilde$ and~$\theta$ denote respectively the Möbius transforms of~$\ftilde$ and of~$\varphi$.

  It follows from point~\ref{item:11} of Lemma~\ref{lem:22} that~$r=\rho$.
Let~$F$ be the square non-negative matrix introduced in~\eqref{eq:63}, and let~$F^+$ be the restriction of~$F$ to stable state-and-cliques.

Consider a basic component of~$F^+$, say~$N$, and some state~$\alpha_0$ such that~$(\alpha_0,d)\in N$ for some clique~$d$.
It follows from Lemma~\ref{lem:22}, point~\ref{item:9}, that the two vectors~$v$ and~$v'$ defined by: 
\begin{align*}
  v_{(\alpha,c)}&=\Gamma(\alpha_0,\alpha)\htilde_\alpha(c)  
  &v'_{(\alpha,c)}&=\Delta(\alpha_0,\alpha)\theta_\alpha(c)
\end{align*}
are positive invariant vectors of~$F^+$.

The component~$N$ is also final in~$F^+$ since~$F^+$ is an umbrella matrix; hence the restrictions of~$v$ and~$v'$ to~$N$ are themselves invariant.
It follows from Perron-Frobenius theory, applied to the irreducible matrix~$N$, that~$v$ and~$v'$ are proportional on~$N$, hence for some constant~$k_{\alpha_0}>0$:
\begin{gather}
\label{eq:28}
    \forall (\alpha,c)\in N\quad \Gamma(\alpha_0,\alpha)\htilde_{\alpha}(c)=k_{\alpha_0}\Delta(\alpha_0,\alpha)\theta_\alpha(c)
  \end{gather}

  In particular, for any clique~$c$ such that~$(\alpha_0,c)\in N$, one has:
  \begin{gather}
\label{eq:29}
    \htilde_{\alpha_0}(c)=k_{\alpha_0}\theta_{\alpha_0}(c)
  \end{gather}

Let~$P$ and~$P'$ be the transitions matrices of the Markov chains of state-and-cliques associated with the probabilistic valuations~$\ftilde$ and~$\varphi$, respectively.
Let a state-and-clique~$(\beta,d)\in N$ and~$c$ a clique such that~$(\alpha_0,c)\in N$.
We start from the expression~\eqref{eq:62} for the transition matrix~$P$, expend the definition of~$\ftilde_{\alpha_0}(c)$, and use~\eqref{eq:28} and~\eqref{eq:29} to find:
\begin{align*}
  P_{(\alpha_0,c),(\beta,d)}
                            &=\un{\beta=\alpha_0\cdot c}\un{c\to d} f_{\alpha_0}(c)r^{\len c}\Delta(\alpha_0,\beta)\frac{\theta_\beta(d)}{\theta_{\alpha_0}(c)}
  =P'_{(\alpha_0,c),(\beta,d)}
\end{align*}

Since~$(\alpha_0,c)$ was arbitrary in~$N$, we conclude that~$P$ and~$P'$ coincide on~$N$.
We claim that:
\begin{compactenum}
\item[(\dag)] This is enough to insure that~$\ftilde=\varphi$.
\end{compactenum}

The claim (\dag) implies that~$\Delta=\Gamma$, hence its proof will complete the proof of Theorem~\ref{thr:2}.
For the proof of~(\dag), let~$(\alpha_0,c_0)\in N$.
In particular,~$(\alpha_0,c_0)$ is a stable state-and-clique, let~$x_0\in\M_{\alpha_0}$ be a protection of~$(\alpha_0,c_0)$.
Now let~$(\alpha,x)$ be any trajectory;
in order to prove that~$\ftilde(\alpha,x)=\varphi(\alpha,x)$, we consider the probability measures~$(\mu_\alpha)_{\alpha\in X}$ and~$(\nu_\alpha)_{\alpha\in X}$ on~$\BM$ associated with~$\ftilde$ and with~$\varphi$, respectively.
Pick~$x_1\in\M_{\alpha_0\cdot x_0}$ such that~$\alpha_0\cdot x_0 x_1=\alpha$.
Then the chain rule~\eqref{eq:10} shows that:
\begin{align}
  \ftilde_\alpha(x)&=
\frac{\mu_{\alpha_0}\bigl(\up(x_0x_1x)\bigr)}{\mu_{\alpha_0}\bigl(\up(x_0x_1)\bigr)}
  \label{eq:30} &
\varphi_\alpha(x)&=
\frac{\nu_{\alpha_0}\bigl(\up(x_0x_1x)\bigr)}{\nu_{\alpha_0}\bigl(\up(x_0x_1)\bigr)}
\end{align}

According to~\eqref{eq:64}, the visual cylinder~$\up(x_0x_1)$ can be described as a disjoint union of standard cylinders.
Hence, denoting by~$C_z$  the standard cylinder associated to~$z$ as in~\eqref{eq:32}, one has:

\begin{gather*}
  \mu_{\alpha_0}\bigl(\up(x_0x_1)\bigr)=\sum_{\substack{z\in\M\tq x_0x_1\leqslant z\\\height z=\height{x_0x_1}}}\mu_{\alpha_0}(C_z)
\end{gather*}

On the one hand,~$x_0$~is a protection of~$(\alpha_0,c_0)$ so all~$z$ such that~$x_0x_1\leqslant z$ have~$(\alpha_0,c_0)$ as their first state-and-clique.
On the other hand,~$N$~is a final component of~$F^+$, so all state-and-cliques of~$z$ either belong to~$N$ or are out of the set of stable state-and-cliques;
in the latter case, the standard cylinder~$C_z$ is given~$\mu_{\alpha_0}$-probability~$0$ by virtue of Lemma~\ref{lem:11} and by the form~\eqref{eq:34} of the transition matrix.
Since the same holds for the probability~$\nu_{\alpha_0}$, and since the transition matrices coincide on~$N$, we conclude that~$\mu_{\alpha_0}\bigl(\up(x_0x_1)\bigr)=\nu_{\alpha_0}\bigl(\up(x_0x_1)\bigr)$.
The same applies to~$\up(x_0x_1x)$, hence from~\eqref{eq:30} we conclude that~$\ftilde_\alpha(x)=\varphi_\alpha(x)$, which was to be proved.
\end{proof}

\subsection{Universal construction of probabilistic valuations and other consequences}
\label{sec:consequences}

We now come to a series of corollaries concerning the structure and the properties of probabilistic valuations, which all derive from the spectral property of Theorem~\ref{thr:1} and from the previous result (Theorem~\ref{thr:2}).
Next section (\S~\ref{sec:kernel-mobius-matrix}) will be devoted to another interesting consequence of Theorem~\ref{thr:2}.

\begin{corollary}
\label{cor:1}
Assume that the system is irreducible and that~$f$ is a probabilistic valuation.
Then~$\rho=1$ and:
\begin{gather}
\label{eq:35}
\forall(\alpha,\beta)\in X\times X\quad \lim_{s\to1^-}\frac{G_\beta(s)}{G_\alpha(s)}=1
\end{gather}
\end{corollary}

\begin{proof}
According to Lemma~\ref{lem:3}, the pair~$(\rho,\Delta)$ makes the positive valuation~$\ftilde$ probabilistic, where~$\ftilde_\alpha(x)=f_\alpha(x)\rho^{\len x}\Delta(\alpha,\alpha\cdot x)$.
But the pair~$(1,1)$ also works since~$f$ is assumed to be probabilistic itself.
Hence, by Theorem~\ref{thr:1},~$\rho=1$ and~$\Delta=1$, which is~\eqref{eq:35}.
\end{proof}

The next result shows that the construction of probabilistic valuations that we have studied so far, and which was introduced in Lemma~\ref{lem:3}, actually covers the range of all possible probabilistic valuations, and is thus universal.

\begin{corollary}
\label{cor:2}
Assume that the concurrent system is irreducible, equipped with a Markov measure~$\nu=(\nu_\alpha)_{\alpha\in X}$.
Then for each~$\alpha\in X$, the probability measure~$\nu_\alpha$ is the weak limit, as~$s\to1^-$, of the probability measures~$(\nu_{\alpha,s})_{0\leqslant s<1}$ on~$\Mbar$ defined as in~\eqref{eq:23}.
\end{corollary}

\begin{proof}
Indeed,~$\rho=1$ by Corollary~\ref{cor:1}, and the weak limit of the statement, which is a Markov measure on~$\BM$ according to Lemma~\ref{lem:3}, coincides with~$\nu_\alpha$ still by the same corollary.
\end{proof}

Recall that, for any letter~$a$,~$\M^a$~denotes the submonoid of~$\M$ generated by~$\Sigma\setminus\{a\}$;
and let:
\begin{gather*}
x\BM_{\alpha\cdot x}^a=\{x\omega\tq \omega\in\BM_{\alpha\cdot x}^a\}
\end{gather*}

The next result shows that, after any finite trajectory, all letters will be almost surely used; this is easily shown to be false if the system is not irreducible, for instance if the underlying trace monoid itself is not irreducible.

\begin{corollary}
\label{cor:7}
Assume that the system is irreducible.
Then for any~$\alpha\in X$, for any letter~$a\in\Sigma$ and for any trace~$x\in\M_\alpha$:\quad~$\nu_\alpha(x\BM_{\alpha\cdot x}^a)=0$.
\end{corollary}

\begin{proof}
It is enough to prove the result for~$x=\ve$, hence that~$\nu_\alpha(\BM_\alpha^a)=0$.

Let~$f=(f_\alpha)_{\alpha\in X}$ be the probabilistic valuation associated to the Markov measure~$\nu$.
Corollary~\ref{cor:1} states that the spectral radius of the generating series~$G_\alpha(t) = \sum_{y\in\M_\alpha}t^{\len y}f_\alpha(y)$ is~$1$, and Theorem~\ref{thr:1} proves that this radius is smaller than the radius of the generating series~$G_\alpha^a(t)=\sum_{y\in\M_\alpha^a}t^{\len y}f_\alpha(y)$; it follows that
\begin{gather}
\label{eq:49}
\sum_{y\in\M_\alpha^a}f_\alpha(y)<\infty.
\end{gather}

Finally, let~$(C_i)_{i\geqslant1}$ be the cliques of a random infinite trace~$\omega$, and consider the events~$Z_n=\{C_1\cdots C_n \in \M^a\}$ for~$n\geqslant0$.
Then, denoting by~$\height y$ the height of a trace~$y$, one has:
\begin{align*}
\sum_{n\geqslant1}\nu_\alpha(Z_n)&=\sum_{n\geqslant1}\;\sum_{y\in\M_\alpha^a}\un{\height y=n}\nu_\alpha\bigl(C_1\cdots C_n = C_1(y) \cdots C_n(y)\bigr)\\
&\leqslant\sum_{n\geqslant1}\;\sum_{y\in\M_\alpha^a}\un{\height y=n}\nu_\alpha(\up y)\\
&\leqslant\sum_{y\in\M_\alpha^a}f_\alpha(y)<\infty\qquad\text{using~\eqref{eq:49}}
\end{align*}
By Borel-Cantelli lemma, it follows that~$\nu_\alpha(Z_n\ \text{infinitely often})=0$, which is another formulation of~$\nu_\alpha(\BM_\alpha^a)=0$.
\end{proof}

The next result gives a probabilistic characterization of stable state-and-cliques. An interesting point is that this characterization does not depend on the particular probabilistic valuation.

\begin{corollary}
\label{cor:3}
Assume that the system is irreducible and that~$f$ is a probabilistic valuation.
Let~$h$ be the Möbius transform of~$f$.
Then any state-and-clique~$(\alpha,c)$ is stable if and only if~$h_\alpha(c)>0$.
In particular, the property~$h_\alpha(c)>0$ for a given state-and-clique is independent of the probabilistic valuation~$f$.
\end{corollary}

\begin{proof}
This follows directly from Lemma~\ref{lem:11}, which applies to~$f$ since in this case~$\ftilde=f$ by Corollary~\ref{cor:1}.
\end{proof}

The next result brings a probabilistic argument for the claim that the unstable state-and-cliques are an ``encoding artefact''. Of course, some trajectories need them to be correctly encoded; but all together, these trajectories have probability zero to occur---independentely of the chosen Markov measure.

\begin{corollary}
\label{cor:8}
Assume that the system is irreducible.
Then, relatively to any Markov measure, the Markov chain of state-and-cliques stays within the \DSCP.
\end{corollary}

\begin{proof}
Derives at once from Corollary~\ref{cor:3} and from the structure of the transition matrix of the chain, given by~\eqref{eq:34}.
\end{proof}

The previous corollary has shown that unstable state-and-cliques ``should not'' occur. On the opposite, some state-and-cliques among the stable ones have the property that they ``must'' occur if they are enabled. The next result shows that the later property does not depend on the chosen Markov measure.

\begin{corollary}
\label{cor:5}
Assume that the system is irreducible, and let\/~$(\alpha,b)\in X\times\Sigma$.
Then the following properties are equivalent:
\begin{compactenum}[\normalfont(i)]
\item\label{item:oihqw}~$f_\alpha(b)=1$ for some probabilistic valuation~$f$;
\item\label{item:qasp}~$f_\alpha(b)=1$ for any probabilistic valuation~$f$.
\end{compactenum}
\end{corollary}

\begin{proof}
In view of Corollary~\ref{cor:3}, the equivalence (\ref{item:oihqw})$\iff$(\ref{item:qasp}) follows at once from the following claim:
\begin{compactitem}
\item[(\dag)] If~$f$ is a probabilistic valuation, then~$f_\alpha(b)=1$ if and only if all state-and-cliques of the form~$(\alpha,\gamma)$ with~$b\notin\gamma$ are not stable.
\end{compactitem}

$(\Rightarrow)$\quad Let~$f$ be a probabilistic valuation such that~$f_\alpha(b)=1$ and let~$(\alpha,\gamma)$ be a state-and-clique such that~$b\notin\gamma$.
Seeking a contradiction, assume that~$(\alpha,\gamma)$ is stable.
Then, with positive~$\nu_\alpha$-probability, the first clique of~$\omega\in\partial\M_\alpha$ is~$\gamma$, and thus~$b$ does not divide~$\omega$.
But this contradicts~$\nu_\alpha(\up b)=f_\alpha(b)=1$.

$(\Leftarrow)$\quad Conversely, assume that all state-and-cliques of the form~$(\alpha,\gamma)$ with~$b\notin\gamma$ are not stable.
Then the first clique of~$\omega\in\partial\M_\alpha$ being stable with~$\nu_\alpha$-probability~$1$,~$b$~divides~$\omega$~$\nu_\alpha$-\as, that is to say~$f_\alpha(b)=\nu_\alpha(\up b)=1$.
\end{proof}

The Markov measures that we study are defined on the space of infinite trajectories. Among them, some are not maximal with respect to the partial order~$\leq$, a property which is specific to concurrent systems: this is due to the fact some branch of the concurrent system could evolve with infinitely many events while leaving another branch steady. But with a probabilistic dynamics, and if
the system is irreducible, it is quite intuitive that this situation should actually not happen; this is the topic of the next result.

\begin{corollary}
\label{cor:4}
Assume that the system is irreducible and let~$\nu=(\nu_\alpha)_{\alpha\in X}$ be a Markov measure.
Then for every state~$\alpha\in X$, the property ``$\omega$~is maximal in the partial order~$(\Mbar,\leqslant)$'' is true~$\nu_\alpha$-almost surely.
\end{corollary}

\begin{proof}
Let~$\omega\in\BM_\alpha$ be not maximal in~$\Mbar$, with~$\omega=(c_i)_{i\geqslant1}$.
Then there is letter a~$a\in\Sigma$ and an integer~$N>1$ such that~$a\notin c_i$ for all~$i\geqslant N$.
Hence, the set~$\Lambda$ of non-maximal~$\omega$ is a subset of the following countable union:
\begin{gather*}
\Lambda\subseteq \bigcup_{x\in\M_\alpha}\bigcup_{a\in\Sigma}x\BM^{a}
\end{gather*}

Each subset~$x\BM^a$ has~$\nu_\alpha$ probability~$0$ according to Corollary~\ref{cor:7};
hence~$\nu_\alpha(\Lambda)=0$.
\end{proof}

\subsection{The kernel of the Möbius matrix}
\label{sec:kernel-mobius-matrix}

Given a valuation~$f$, we seek an effective way to determine the unique positive real~$r$ and the unique positive cocycle~$\Delta$ such that the transformed valuation~$\ftilde$ defined by~$\ftilde_\alpha(x)=r^{\len x}\Delta(\alpha,\alpha\cdot x)f_\alpha(x)$ is probabilistic.
It turns out that both of them are closely related to the Möbius matrix~$\Mob(t)$ introduced in~\S~\ref{sec:concurrent-systems}.
Recall that~$\Mob(t)$ is the polynomial matrix of size~$X\times X$ defined by~$\Mob_{\alpha,\beta}(t)=\sum_{c\in\C_{\alpha,\beta}}f_\alpha(c)(-1)^{\len c}t^{\len c}$.

Let the polynomial~$\theta(t)$ be defined by:
\begin{gather*}
\theta(t)=\det\Mob(t)
\end{gather*}

In the next result, we prove in particular that the kernel of the Möbius matrix $\Mob(\rho)$ evaluated at the root $\rho$ of the system, has dimension~$1$, which is a non trivial result. As a consequence, we derive a practical way---at least for small examples---to determine the positive cocycle normalizing a given valuation to a probabilistic valuation. See an example of use of this technique below.

\begin{theorem}
\label{thr:3}
Assume that the concurrent system is irreducible, equipped with a valuation~$f$.
Let~$\rho>0$ and\/~$\Delta:X\times X\to \RR_{>0}$ be the positive real and the positive cocycle making~$\ftilde$ probabilistic.
Then:
\begin{compactenum}
\item\label{item:7}~$\rho$ is the positive root of\/~$\theta(t)$ of smallest modulus among its complex roots.
\item\label{item:8}~$\dim\bigl(\ker\Mob(\rho)\bigr)=1$ and\/~$\ker\Mob(\rho)$ is generated by a positive vector~$U$.
Furthermore:
$\Delta(\alpha,\beta)=U_\beta/U_\alpha$.
\end{compactenum}
\end{theorem}

\begin{proof}
\ref{item:7}.\quad Since~$\rho$ is the radius of convergence of the matrix series~$G(t)$ with non-negative coefficients, and since~$G(t)\Mob(t)=\Id$,~$\rho$~is indeed a root of smallest modulus of~$\theta(t)$.

\ref{item:8}.\quad We shorten the notation by putting~$\Mob=\Mob(\rho)$.
Let~$\htilde$ be the Möbius transform of~$\ftilde$.
Pick an arbitrary state~$\alpha_0\in X$ and put~$U_\alpha=\Delta(\alpha_0,\alpha)$ for~$\alpha\in X$.
Using the cocycle property of~$\Delta$ and the definition of~$\Mob$ yields:
\begin{gather*}
(\Mob U)_\alpha=\Delta(\alpha_0,\alpha)\sum_{\beta\in X}\sum_{c\in\C_{\alpha,\beta}}(-1)^{\len c}\rho^{\len c}f_\alpha(c)\Delta(\alpha,\beta)=\Delta(\alpha_0,\alpha)\htilde_\alpha(\ve)=0
\end{gather*}
Hence, the vector~$U=(U_\alpha)_{\alpha\in X}$ is a positive null vector of~$\Mob$.

Seeking a contradiction, assume that~$\dim(\ker\Mob)>1$.
Let~$W$ be a null vector of~$\Mob$ not proportional to~$U$, and let~$V=U+yW$ where~$y$ is a real small enough such that~$V>0$;
let~$\Gamma$ be the cocycle~$\Gamma(\alpha,\beta)=V_\beta/V_\alpha$, let~$\varphi$ be the valuation~$\varphi_\alpha(x)=\rho^{\len x}\Gamma(\alpha,\alpha\cdot x)f_\alpha(x)$ and let~$\psi$ be the Möbius transform of~$\varphi$.
Then the definition of the Möbius transform yields, for every~$\alpha\in X$ and~$c\in\C_\alpha$:
\begin{gather}
\label{eq:37}
\psi_\alpha(c)=\frac1{V_\alpha}\Bigl(\Delta(\alpha_0,\alpha)\htilde_\alpha(c)+y
\sum_{c'\in\C\tq c\leqslant c'}(-1)^{\len {c'}-\len c}\rho^{\len {c'}}W_{\alpha\cdot c'}f_\alpha(c')\Bigr)
\end{gather}
We claim that:
\begin{compactenum}
\item[(\ddag)] For~$y$ small enough,~$\psi$~is a probabilistic valuation.
\end{compactenum}
Firstly, from~\eqref{eq:37} follows~$\psi_\alpha(\ve)=(\Delta(\alpha_0,\alpha)\htilde_\alpha(\ve)+y(\Mob W)_\alpha)/V_\alpha=0$.
Secondly, if~$(\alpha,c)$ is a stable state-and-clique, then~$\htilde_\alpha(c)>0$ according to Lemma~\ref{lem:11}, and thus~\eqref{eq:37} implies that~$\psi_\alpha(c)>0$ for~$y$ small enough.
Thirdly, we prove that~$\psi_\alpha(c)=0$ if~$(\alpha,c)$ is not a stable state-and-clique.
Lemma~\ref{lem:8} applies to the valuation~$\varphi$ since~$\psi_\alpha(\ve)=0$ for all~$\alpha\in X$.
Therefore,  the vector~$v_{(\alpha,c)}=\Gamma(\alpha_0,\alpha)\psi_\alpha(c)$ satisfies~$Fv=v$, where~$F$ is the matrix defined by~$F_{(\alpha,c),(\beta,d)}=\rho^{\len c}f_\alpha(c)$.

We know from Lemma~\ref{lem:22} that~$F$ has the block shape~$F=\left(\begin{smallmatrix}
F^+ &X\\0&F^0
\end{smallmatrix}
\right)$ where stable state-and-cliques are put first.
Hence, if~$v=\left(\begin{smallmatrix}
v^+\\v^0
\end{smallmatrix}
\right)$ is the corresponding decomposition of~$v$ then~$F^0v^0=v^0$;
but~$F^0$ has spectral radius~$<1$ hence~$v^0=0$ and thus~$\psi_\alpha(c)=0$ if~$(\alpha,c)$ is not a stable state-and-clique, which was to be proved.

We have obtained that~$\psi_\alpha(\ve)=0$ for all~$\alpha\in X$ and that~$\psi_\alpha(c)\geqslant0$ for all state-and-cliques~$(\alpha,c)$;
these are the conditions~\eqref{eq:11} to ensure that~$\psi$ is a probabilistic valuation, proving the claim~(\ddag).

By the uniqueness statement of Theorem~\ref{thr:2}, we conclude that~$\Delta=\Gamma$, but this contradicts that~$W$ has been chosen not proportional to~$U$, and thus~$\dim(\ker\Mob)=1$.

Since~$U_\alpha=\Delta(\alpha_0,\alpha)$, the cocycle property of~$\Delta$ yields~$\Delta(\alpha,\beta)=U_\beta/U_\alpha$, as claimed.
\end{proof}

\paragraph{Example.}

Let us give a new look at the example introduced in~\S~\ref{sec:concurrent-systems}, the action of which is depicted on Figure~\ref{fig:qwqwdqaazaxcpow}.
Starting from the counting valuation, the Möbius matrix, indexed by the states~$0$,~$1$~and~$2$ in this order, is given by:
\begin{gather*}
\Mob(t)=
\begin{pmatrix}
1&-t&-t+t^2\\
-t&1&-t+t^2\\
-t&0&1-2t+t^2
\end{pmatrix}
\quad\theta(t)=\det\Mob(t)=(1-t^2)(1-2t)
\end{gather*}
hence~$\rho=\frac12$.
Then~$\ker\Mob(\frac12)$ is generated by the positive vector~$U=\left(
\begin{smallmatrix}
1\\1\\2
\end{smallmatrix}
\right)$, which determines the cocycle~$\Delta$.
The computation yields the values for~$\lambda_\alpha(x)=\rho^{\len x}\Delta(\alpha,\alpha\cdot x)$ given in~\eqref{eq:36} for~$x\in\Sigma$.

In this example, the \DSC\ contains 10 state-and-cliques, already depicted on Figure~\ref{fig:qwdqqwfgkojngfffff}.
Two of them are not stable:
$(0,a)$ and~$(1,a)$.
Indeed, with~$h$ the Möbius transform of~$\lambda$:
\begin{align*}
h_0(a)&=\lambda_0(a)-\lambda_0(a)\lambda_1(b)=0
&h_1(a)&=\lambda_1(a)-\lambda_1(a)\lambda_0(b)=0
\end{align*}

The \DSCP\ contains the 8 remaining state-and-cliques, all stable.
It has a unique basic component with 6 state-and-cliques: see Figure~\ref{fig:asdopijqa}.

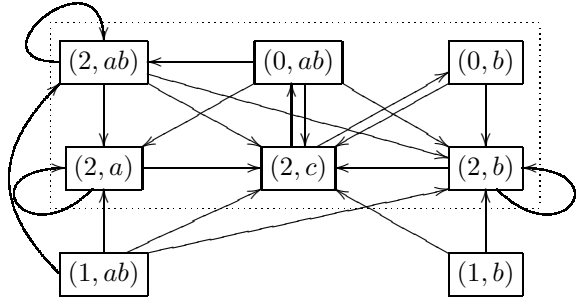
\begin{figure}
  \begin{gather*}
        \xymatrix@C=4em{
      *+[F]{(2,ab)}\POS!L\ar@(l,u)[]!U\POS[]\ar[d]\POS!D!R\ar[dr]!U!L\POS[]\ar[drr]&
      *+[F]{(0,ab)}\ar[l]\POS!D!L\ar[dl]!U!R\POS[]\ar@<.6ex>[d]\ar@<-.6ex>@{<-}[d]\POS!D!R\ar[dr]!U!L&
      *+[F]{(0,b)}\ar[d]\POS!D!L\ar[dl]!U!R\POS[]\POS!L!D(.5)\ar@{<-}[dl]!U!R(.5)\\
      *+[F]{(2,a)}\ar[r]\POS!D!L(.3)\ar@(dl,l)[]!L&
      *+[F]{(2,c)}&
      *+[F]{(2,b)}\ar[l]\POS!D!R(.3)\ar@(dr,r)[]!R\\
      *+[F]{(1,ab)}\ar[u]\POS!U!R(.5)\ar[ur]!D!L\POS!U!R\ar[urr]!D!L\POS[]\POS!L\ar@(ul,dl)[uu]!D!L&
      &
      *+[F]{(1,b)}\ar[u]\ar[ul]!D!R
      \save"1,1"+<-2em,1.5em>.{"2,3"+<2em,-1.5em>}*[F.]\frm{}\restore
}
\end{gather*}
\caption{The \DSCP\ of our running example.
The single basic component is framed.
Compare with the \DSC\ depicted on Figure~\ref{fig:qwdqqwfgkojngfffff}.}
\label{fig:asdopijqa}
\end{figure}

\begin{remark}
Since the valuation~$\lambda$ previously defined in~\eqref{eq:36} is probabilistic, and since~$\lambda_0(b)=1$ and~$\lambda_1(b)=1$, it follows from Corollary~\ref{cor:5} that~$f_0(b)=1$ and~$f_1(b)=1$ both hold for \emph{any} probabilistic valuation~$f$.
\end{remark}

\section{Ergodic properties of irreducible concurrent systems}
\label{sec:ergod-prop-irred}

\subsection{An ergodic property}
\label{sec:an-ergod-prop}

We now come to the study of the ergodic properties of irreducible concurrent systems.
Since the Markov chain of state-and-cliques is not ergodic in general, as its digraph may not be strongly connected, the ergodic properties must rely on the intrinsic concurrent systems itself, and not on its combinatorial representation.

At the scale of the intrinsic system, there is however no global time available.
Henceforth the standard formulation of ergodicity does not apply in a straightforward way, whence the following definition.

\begin{definition}
Let~$\SS=(\M,X)$ be a concurrent system equipped with a Markov measure~$\nu$, and let~$Z=(Z_\alpha)_{\alpha \in X}$ be a family of measurable functions~$Z_\alpha:\partial\M_\alpha\to X$.
It is \emph{shift-invariant} if for all~$\alpha\in X$ and for~$\nu_\alpha$-\as\ every~$\omega\in\BM_\alpha$:
\begin{gather}
\label{eq:38}
\forall x\in\M_\alpha\quad x\leqslant\omega\implies Z_\alpha(\omega)=Z_{\alpha\cdot x}(x^{-1}\omega)
\end{gather}
where~$x^{-1}\omega$ is the unique~$\xi\in\BM_{\alpha\cdot x}$ such that~$\omega=x\xi$.
\end{definition}

\begin{theorem}
\label{thr:4}
Let~$\SS=(\M,X)$ be an irreducible concurrent system equipped with a Markov measure~$\nu$.
If\/~$Z=(Z_\alpha)_{\alpha\in X}$ is a family of non-negative or bounded functions, and if\/~$Z$ is shift-invariant, then there is a constant~$z$ such that\/~$Z_\alpha(\omega)=z$ for all~$\alpha\in X$ and for~$\nu_\alpha$-\as\ every~$\omega\in\BM_\alpha$.
\end{theorem}

\begin{proof}
The proof proceeds in three parts.
We first establish a theory of stopping times for concurrent systems, yielding to a Strong Markov property for concurrent systems.
The second part introduces hitting times and a new Markov chain attached to the irreducible concurrent system.
The last part applies these results to conclude the proof.

\emph{First part:
stopping times and the Strong Markov property.}\quad Define a \emph{stopping time} as a family~$T=(T_\alpha)_{\alpha\in X}$ of mappings~$T_\alpha:\BM_\alpha\to\Mbar_\alpha$ such that, for all~$\alpha\in X$:
\begin{gather}
\label{eq:39}
\forall \omega\in\BM_\alpha\quad T_\alpha(\omega)\leqslant\omega\\
\label{eq:40}
\forall \omega,\omega'\in\BM_\alpha\quad T_\alpha(\omega)\leqslant\omega'\implies T_\alpha(\omega')=T_\alpha(\omega)
\end{gather}
We say that~$T$ is~$\nu$-\as\ finite if~$T_\alpha(\omega)\in\M_\alpha$ for~$\nu_\alpha$-\as\ every~$\omega\in\BM_\alpha$.

If~$T$ is~$\nu$-\as\ finite, the \emph{shift operator} associated to~$T$ is the family of mappings~$\theta_T$ defined, for~$\alpha\in X$ and~$\omega\in\BM_\alpha$, by~$\theta_T(\omega)=(T(\omega))^{-1}\omega$, where the dependence to~$\alpha$ is dropped to shorten the notation.

If~$S$ and~$T$ are two stopping times, the stopping time~$S\circ T$ is defined, for~$\alpha\in X$, on~$\{T_\alpha(\omega)\in\M_\alpha\}$, by:
\begin{gather}
\label{eq:41}
(S\circ T)_\alpha(\omega)=T_\alpha(\omega)S_{\alpha\cdot T(\omega)}\bigl(\theta_T(\omega)\bigr)
\end{gather}
In particular, the iterates~$T^n$ of a stopping time~$T$ are defined, for all~$n\geqslant0$, by:
\begin{gather}
\label{eq:42}
T^0=\ve,\qquad\forall n\geqslant0\quad T^{n+1}=T\circ T^n
\end{gather}

If~$T$ is a stopping time, let~$\T_\alpha=\{T_\alpha(\omega)\tq \omega\in\BM_\alpha\}$.
The family of \emph{\slgb s associated with~$T$} is the family~$\F=(\F_\alpha)_{\alpha\in X}$ where~$\F_\alpha$ is the \slgb\ generated by the at most countable collection of visual cylinders of the form~$\up x$, for~$x$ ranging over~$\T_\alpha\cap\M_\alpha$.

The \emph{Strong Markov property} then states as follows:
if~$Z=(Z_\alpha)_{\alpha\in X}$ is a collection of mappings~$Z_\alpha\to\RR$, either bounded or non-negative, and if~$T$ is a stopping time, then:
\begin{gather}
\label{eq:43}
\forall\alpha\in X\quad\text{$\nu_\alpha$-\as\quad} \esp_\alpha(Z_{\alpha\cdot T_\alpha}\circ T_\alpha|\F_\alpha)=\esp_{T_\alpha}(Z_{\alpha\cdot T_\alpha})
\end{gather}
where~$\esp_\alpha$ denotes the expectation with respect to~$\nu_\alpha$, and where both members are conventionally~$0$ where~$T_\alpha$ is not finite.
The proof is straightforward from the definitions, hence we omit it.

\medskip
\emph{Second part:
hitting times.}\quad In Markov chain theory, one is accustomed with the notion of hitting time of a state.
Here, there is no natural notion of ``first hitting time of a state'';
for, if~$x,y\leqslant\omega$ are such that~$\alpha\cdot x=\beta$ and~$\alpha\cdot y=\beta$, nothing ensures that~$x\wedge y$ would satisfy~$\alpha\cdot (x\wedge y)=\beta$.
Hence, there is no well-defined notion of ``first hitting'' a given state.
By contrast, hitting a letter of the base alphabet of the monoid is a well-defined notion.

Let~$a\in\Sigma$.
For~$\omega\in\partial\M_\alpha$, if there is at least some~$x\leqslant\omega$ with an occurrence of~$a$, then we define the first hitting time of~$a$ by:
\begin{gather}
\label{eq:44}
T^a_\alpha(\omega)=\min\{x\in\M_\alpha\tq x\leqslant\omega\text{ and }\len x_a>0\}
\end{gather}
where~$\len x_a$ denotes the number of occurrences of~$a$ in~$x$, which is indeed a trace quantity;
Then~$T^a=(T^a_\alpha)_{\alpha\in X}$ is a stopping time in the sense of~\eqref{eq:39} and~\eqref{eq:40}.

\begin{lemma}
\label{lem:5}
Assume that~$\SS$ is irreducible and equipped with a Markov measure~$\nu$.
For any~$a\in\Sigma$,~$T^a$~and all its iterates~$(T^a)^n$ are~$\nu$-\as\ finite, and:
\begin{gather*}
  \forall\alpha\in X\quad\forall k\geqslant0\quad\esp_\alpha\bigl[\len{T^a_\alpha}^k\bigr]<\infty
\end{gather*}

Furthermore, if~$\F^n$ is the \slgb\ associated with~$(T^a)^n$, then:
\begin{gather}
\label{eq:45}
\bigvee_{n\geqslant0}\F^n=\F
\end{gather}
And for every~$\alpha\in X$ and for~$\nu_\alpha$-\as\ every~$\omega\in\BM_\alpha$:
\begin{gather}
\label{eq:51}
\omega=\bigvee\bigl\{(T^a)^n(\omega)\tq n\geqslant0\bigr\}
\end{gather}
\end{lemma}

\begin{proof}
If~$T_\alpha^a(\omega)\notin\M$ then~$a$ does not occur in~$\omega$.
Hence,~$\{T_\alpha^a\notin\M\}\subseteq\BM_\alpha^a$, but~$\nu_\alpha(\BM^a)=0$ by Corollary~\ref{cor:7}, hence~$T_\alpha^a$ is~$\nu_\alpha$-\as\ finite.
It follows by induction, and with the Strong Markov property~\eqref{eq:43}, that all the iterates~$(T_\alpha^a)^n$ are~$\nu_\alpha$-\as\ finite as well.

Let~$\alpha\in X$, and let~$\Omega\subseteq\BM_\alpha$ such that~$\nu_\alpha(\Omega)=1$, and~$T_\alpha^a$ and all its iterates are finite on~$\Omega$.
Define~$\T^a_\alpha=\{T^a_\alpha(\omega)\tq\omega\in\Omega\}$, the set of finite values of~$T_\alpha^a$ on~$\BM_\alpha$.
Then the stopping time properties~\eqref{eq:39} and~\eqref{eq:40} imply that~$\Omega$ can be written as the disjoint union:
\begin{gather*}
  \Omega=\bigsqcup_{x\in\T^a_\alpha}\up x
\end{gather*}
Henceforth, for every integer~$k\geqslant0$, the~$k^\text{th}$ moment of~$\len{T^a_\alpha}$ is:
\begin{gather}
\label{eq:68}
\esp_\alpha\bigl[\len{T^a_\alpha}  ^k\bigr]=\sum_{x\in\T^a_\alpha}{\len x}^kf_\alpha(x)
\end{gather}
where~$f=(f_\alpha)_{\alpha\in X}$ is the probabilistic valuation associated to the Markov measure~$\nu$.
Every~$x\in\T_\alpha^a$ writes as~$x=ya$ with~$y\in\M_\alpha^a$.
Therefore~\eqref{eq:68} implies:
\begin{gather*}
  \esp_\alpha\bigl[\len{T^a_\alpha}  ^k\bigr]\leqslant\sum_{y\in\M_\alpha^a}(1+\len y)^kf_\alpha(y)<\infty
\end{gather*}

Indeed, the series converges by the following argument:
\begin{inparaenum}[1:]
\item the radius of convergence of the series~$\sum_{y\in\M_\alpha}s^{\len y}f_\alpha(y)$ is~$1$ by Corollary~\ref{cor:1}; and
\item the radius of convergence of the series~$\sum_{y\in\M_\alpha^a}s^{\len y}f_\alpha(y)$ is thus greater than~$1$, thanks to Theorem~\ref{thr:1}.
\end{inparaenum}

We now prove~\eqref{eq:51}.
For~$\omega\in\Omega$, let~$\xi(\omega)=\bigvee_{n\geqslant0}(T_\alpha^a )^n(\omega)$, which is well defined thanks to Proposition~\ref{prop:3}, point~\ref{item:4987}.
We wish to prove that~$\xi=\omega$ for~$\nu_\alpha$-\as\ every~$\omega\in\Omega$.
Let~$N=\inf\{n\geqslant0\tq C_n(\omega)\neq C_n(\xi)\}$.
Then~$N<\infty$ on~$\{\xi\neq\omega\}$, and there is a letter~$b\in C_N(\omega)$ such that~$b\notin C_N(\xi)$.
We claim that:
\begin{gather}
  \label{eq:74}
  \text{~$b\notin C_n(\xi)$ for all~$n\geqslant N$}
\end{gather}

Indeed, let~$n\geqslant N$.
Referring to the definition~\eqref{eq:73} of~$\xi\leqslant\omega$, one has~$C_1(\xi)\cdots C_n(\xi)\leqslant C_1(\omega)\cdots C_n(\omega)$, and thus:
\begin{gather*}
C_{N}(\xi)\cdots C_n(\xi)\leqslant C_{N}(\omega)\cdots C_n(\omega)  
\end{gather*}
From the characterization~\eqref{eq:75} of the order~$(\M,\leqslant)$ follows the implication~$b\notin C_N(\omega)\implies b\notin C_{N+1}(\omega),\ldots, C_n(\omega)$, which proves~\eqref{eq:74}.
Consequently:

\begin{gather*}
\{\xi\neq\omega\}\subseteq\bigcup_{(x,b)\in\M_\alpha\times \Sigma} x\BM_{\alpha\cdot x}^b
\end{gather*}
a countable collection of~$\nu_\alpha$-null sets according to Corollary~\ref{cor:7}, and thus~$\xi=\omega$,~$\nu_{\alpha}$-\as, which proves~\eqref{eq:51}.

Finally, to prove~\eqref{eq:45}, let~$x\in\M_\alpha$.
Define the stopping time~$S_\alpha$ by
\begin{gather*}
  R_\alpha(\omega)=(T^a)^{N(\omega)}(\omega),\quad N(\omega)=\inf\{n\geqslant0\tq x\leqslant (T^a)^n(\omega)\}
\end{gather*}
if~$N(\omega)<\infty$, and by~$R_\alpha(\omega)=\omega$ otherwise.
Then~\eqref{eq:51} shows that~$N(\omega)<\infty$ for~$\nu_\alpha$-\as\ every~$\omega\in\up x$.
Therefore~$\up x$ can be written as a countable and disjoint union of the form:
\begin{gather*}
  \up x=\bigsqcup_{y\in\R_\alpha^a}\up y
\end{gather*}
where~$\R^a_\alpha=\{R_\alpha(\omega)\tq\omega\in\BM_\alpha,\quad N(\omega)<\infty\}$.
Hence,~$\up x\in\bigvee_{n\geqslant0}\F_\alpha^n$ and thus~$\F=\bigvee_{n\geqslant0}\F^n$ since~$\F_\alpha=\sigma\langle\up x\tq x\in\M_\alpha\rangle$.

\end{proof}

We now introduce a new Markov chain induced by the system.
With the letter~$a\in\Sigma$ still fixed, pick an initial state~$\alpha\in X$ and let~$\seq Xn$ be defined by:
\begin{gather}
\label{eq:53}
\forall n\geqslant0\quad X_{n}=\alpha\cdot (T^a)^n
\end{gather}
which is~$\nu_\alpha$-\as\ well defined.
In particular~$X_0=\alpha$, and then~$X_n$ is the state reached by the process at its~$n^{\text{th}}$~hitting of the letter~$a$.

\begin{lemma}
\label{lem:4}
The process~$\seq Xn$ defined by~\eqref{eq:53} is a Markov chain with a unique closed class.
\end{lemma}

\begin{proof}
For~$\alpha,\beta\in X$, let~$\T_\alpha^\beta=\{T^a(\omega)\tq \omega\in\BM_\alpha,\ \alpha\cdot T^a(\omega)=\beta\}$.
For any integer~$n\geqslant0$ and any sequence~$x_0=\alpha$ and~$x_1,\ldots,x_{n+1}\in X$, we have:
\begin{align*}
\nu_\alpha(X_0=x_0,\ldots,X_{n+1}=x_{n+1})&=\sum_{\substack{t_i\in\T_{x_{i-1}}^{x_i},\;1\leqslant i\leqslant n+1}}\nu_\alpha\bigl(\up(t_1\cdots t_{n+1})\bigr)\\
&=\sum_{\substack{t_i\in\T_{x_{i-1}}^{x_i},\;1\leqslant i\leqslant n+1}}\nu_\alpha\bigl(\up(t_1\cdots t_{n})\bigr)\nu_{x_n}(\up t_{n+1})\\
&=\nu_\alpha(X_0=x_0,\ldots,X_{n}=x_{n})\nu_{x_n}(X_1=x_{n+1})
\end{align*}
which proves that~$\seq Xn$ is a Markov chain, with transition probability from~$\alpha$ to~$\beta$ given by~$\nu_\alpha(X_1=\beta)=\nu_\alpha(\alpha\cdot T_\alpha^a=\beta)$.

To prove that the chain~$\seq Xn$ has a unique closed class, it is enough to prove the following confluence property:
for every two states~$\alpha,\beta\in X$, there is a state~$\gamma\in X$ such that both~$\alpha$ and~$\beta$ lead to~$\gamma$ in the graph of the chain.
Indeed, let~$x\in\M_\beta$ such that~$\beta\cdot x=\alpha$.
It follows from~\eqref{eq:51} that there is a trace~$y\in\M_\beta$ of the form~$y=(T^a)^n$ such that~$x\leqslant y$.
And then~$\gamma=\beta\cdot y$ is suitable since~$x^{-1}y\in\M_\alpha$ is also of the form~$(T^a)^m$.
\end{proof}

\begin{corollary}
\label{cor:6}
For every letter~$a\in\Sigma$, there exists at least a state~$\alpha_0\in X$ such that, for every~$\alpha\in X$ and for~$\nu_\alpha$-\as\ every~$\omega\in\BM_\alpha$, the sequence~$((T^a)^n)_{n\geqslant0}$ satisfies~$\alpha\cdot (T^a)^n=\alpha_0$ for infinitely many~$n\geqslant0$.
\end{corollary}

\begin{proof}
The Markov chain~$\seq Xn$ associated to~$a$ has a unique closed class according to Lemma~\ref{lem:4}.
Then any state~$\alpha_0$ of this class is recurrent and is reached infinitely often by \as\ every trajectory, hence it is suitable.
\end{proof}

\emph{Third part:
conclusion.}\quad
Assume that~$Z=(Z_\alpha)_{\alpha\in X}$ is a shift-invariant family of non-negative or bounded random variables.
Pick~$a\in\Sigma$ an arbitrary letter and let~$((T^a)^n)_{n\geqslant0}$ be the sequence of iterated hitting times of~$a$.
Let also~$\alpha_0\in X$ be chosen according to Corollary~\ref{cor:6}, and introduce the stopping time
\begin{gather*}
S(\omega)=(T^a)^{N(\omega)}(\omega)\qquad\text{with }N(\omega)=\inf\{n\geqslant0\tq\alpha\cdot (T^a)^n(\omega)=\alpha_0\}
\end{gather*}
It follows from Corollary~\ref{cor:6} that~$S$ and all its iterates~$S^n$ are~$\nu_\alpha$-\as\ finite.

Let~$\G_n$ be the \slgb\ associated to the stopping time~$S^n$.
By the Strong Markov property~\eqref{eq:43} on the one hand, and by using the shift invariance of~$Z$ on the other hand, one has:
\begin{gather}
\label{eq:46}
\text{$\nu_\alpha$-\as}\quad\esp_{\alpha}(Z_{\alpha}|
\G_n)=\esp_{\alpha\cdot S^n}(Z_{\alpha\cdot S^n})
\end{gather}
And since~$\alpha\cdot S^n=\alpha_0$ by construction of~$S$, we obtain~$\esp_\alpha(Z_\alpha|\G_n)=\esp_{\alpha_0}(Z_{\alpha_0})$.
Since~$\bigvee\G_n=\F$ as a consequence of~\eqref{eq:45}, the martingale convergence theorem implies~$Z_\alpha=\esp_{\alpha_0}(Z_{\alpha_0})$,~$\nu_\alpha$-\as, completing the proof.
\end{proof}

\subsection{A Strong Law of Large Numbers}
\label{sec:law-large-numbers}

Let us call a \emph{test function} any family of functions~$\varphi=(\varphi_\alpha)_{\alpha\in X}$ with~$\varphi_\alpha:\M_\alpha\to \RR$ such that, either:
\begin{compactenum}
\item~$\varphi$ is additive, meaning:\quad~$\forall\alpha\in X\quad\forall x\in\M_\alpha\quad\forall y\in\M_{\alpha\cdot x}\quad\varphi_\alpha(xy)=\varphi_\alpha(x)+\varphi_{\alpha\cdot x}(y)$;
or
\item~$\varphi$ is non-negative, non-decreasing and sub-additive, the latter meaning, for all~$\alpha\in X$,~$x\in\M_\alpha$ and~$y\in\M_{\alpha\cdot x}$:
  \begin{gather*}
\varphi_\alpha(x)\leqslant\varphi_\alpha(xy)\leqslant\varphi_\alpha(x)+\varphi_{\alpha\cdot x}(y)    
  \end{gather*}
\end{compactenum}

Any additive function~$\varphi$ is entirely determined by the values~$\varphi_\alpha(a)$ for~$\alpha\in X$ and~$a\in\Sigma$ such that~$\alpha\cdot a\neq\bot$, and provided that~$\varphi_\alpha(a)+\varphi_{\alpha\cdot a}(b)=\varphi_\alpha(b)+\varphi_{\alpha\cdot b}(a)$ whenever~$ab\in\M_\alpha$ and~$ab=ba$.
An interesting example of a sub-additive function is the height function~$\varphi_\alpha(x)=\height x$.

The \emph{ergodic means} associated to a test function~$\varphi$ are defined, for any non-empty trace~$x\in\M_\alpha$, by:
\begin{gather*}
M\varphi(x)=\frac{\varphi_\alpha(x)}{\len x}
\end{gather*}

We say that the ergodic means associated to a test function~$\varphi$ \emph{converge \as\ toward~$L=(L_\alpha)_{\alpha\in X}$} if, for every~$\alpha\in X$:
\begin{gather*}
\text{$\nu_\alpha$-\as}\quad\lim_{x\in\M,\ x\to\omega} M\varphi(x)=L_\alpha(\omega)
\end{gather*}
To make the definition of the limit explicit:
\begin{gather}
  \label{eq:59}
  \forall\ve>0\quad\exists x\in\M_\alpha\quad\forall y\in\M_\alpha\quad(x\leqslant y\leqslant\omega\implies|M\varphi(y)-L_\alpha(\omega)|<\ve)
\end{gather}

In particular, for every sequence~$\seq xn$ of traces converging in~$\Mbar$ toward~$\omega$, one has~$M\varphi(x_n)\xrightarrow{n\to\infty}L_\alpha(\omega)$.
Indeed, for a given~$\ve>0$ and~$x\in\M_\alpha$ as in~\eqref{eq:59}, there exists an integer~$N$ such that~$x\leqslant x_n$ for all~$n\geqslant N$.

\begin{theorem}
\label{thr:6}
Assume that the concurrent system is irreducible and equipped with a Markov measure~$\nu$.
\begin{enumerate}
\item\label{item:14} For every test function~$\varphi$, there is a constant~$k_\varphi$ such that the ergodic means~$M\varphi$ converge \as\ toward the constant~$L=(L_\alpha)_{\alpha\in X}$ where~$L_\alpha=k_\varphi$ for all~$\alpha\in X$.
\item\label{item:15} If~$\varphi$ is additive, the limit~$k_\varphi$ can be computed as follows.
Let~$J$ be any final component of\/ \DSCP.
Let~$Q$ be the restriction to~$J$ of the transition matrix of the Markov chain of state-and-cliques, and let~$\pi_J$ be the unique invariant probability distribution of~$Q$.
Then:
\begin{gather}
  \label{eq:69}
  k_\varphi=\Bigl(\sum_{(\alpha,c)\in J}\pi_J(\alpha,c)\,\len c
  \Bigr)^{-1}\sum_{(\alpha,c)\in J}\pi_J(\alpha,c)\varphi_\alpha(c)
\end{gather}
\end{enumerate}
\end{theorem}

We first establish three lemmas.

\begin{lemma}
  \label{lem:6}
  Let~$\varphi$ be a test function.
For any letter~$a\in\Sigma$, there is a constant~$\ell\in\RR$ such that, for all~$\alpha\in X$, the sequence~$\bigl(M\varphi((T^a)^n)\bigr)_{n\geqslant0}$ converges~$\nu_\alpha$-\as\ toward~$\ell$.
\end{lemma}

\begin{proof}
By construction,~$(T^a)^n$ can be written as a concatenation of traces:
$(T^a)^n=\Delta_1\cdots\Delta_n$ with~$\Delta_i=T^a(\theta_{(T^a)^{i-1}})$, and thus:
\begin{gather}
\label{eq:47}
M\varphi((T^a)^n)=
\left.\Bigl(\frac{\varphi((T^a)^n)}{n}\Bigr)\middle/
\Bigl(\frac {\len{\Delta_1}+\dots+\len{\Delta_n}}n\Bigr)\right.
\end{gather}
Since~$X_n=\alpha\cdot(T^a)^n$ is a Markov chain according to Lemma~\ref{lem:4}, so is~$(X_n,\Delta_{n+1})$.
Therefore both factors in the right hand member of~\eqref{eq:47} converge, according to Kingman ergodic theorem in case where~$\varphi$ is sub-additive.
But the limit is shift-invariant: indeed it a Cesaro-like limit; hence it is constant thanks to Theorem~\ref{thr:4}.
\end{proof}

\begin{lemma}
\label{lem:7}
For~$\omega\in\BM_\alpha$, let~$U_\alpha^a(\omega)=\M_\alpha^a\wedge\omega$ be the largest left divisor of~$\omega$ with no occurrence of~$a$:
\begin{gather*}
  \M_\alpha^a\wedge\omega=\bigvee\{x\in\M_\alpha^a\tq x\leqslant\omega\}
\end{gather*}
Then\/~$\esp_\alpha(\len {U_\alpha^a}^k)<\infty$ for every~$\alpha\in X$ and for every integer~$k\geqslant0$.
\end{lemma}

\begin{proof}
Let~$f_\alpha(x)=\nu_\alpha(\up x)$ be the probabilistic valuation associated to the Markov measure~$\nu$.
For~$\alpha\in X$, consider the generating series
\begin{gather*}
S_\alpha(t)=\sum_{y\in\M_\alpha^a}t^{|y|}f_\alpha(y)
\end{gather*}
where the sum is taken over traces~$y$ that do not have any occurrences of~$a$.
It follows from Theorem~\ref{thr:1} and Corollary~\ref{cor:1} that the radius of convergence of~$S_\alpha(t)$ is greater than~$1$.
Henceforth, for all~$\alpha\in X$ and all~$k > 0$:
\begin{gather}
\label{eq:48}
\sum_{y\in\M_\alpha^a}|y|^k f_\alpha(y)<\infty.
\end{gather}

Now let~$\T_\alpha^a$ denote the set of finite values taken by the first hitting time of~$a$, starting from~$\alpha$.
On the one hand,~$\len{T_\alpha^a}<\infty$~$\nu_\alpha$-\as\ according to Lemma~\ref{lem:5}; on the other hand, for every~$x\in\T_\alpha^a$, the stopping time property of~$T^a_\alpha$ implies~$\{T^a_\alpha=x\}=\up x$.
Therefore, decomposing according to the values of~$T_\alpha^a$ yields:
\begin{align*}
\esp_\alpha\bigl(\len {U_\alpha^a}^k\bigr)&=\sum_{x\in\T_\alpha^a}\esp_\alpha\bigl(\mathbf{1}_{\up\, x}\,\len {U_\alpha^a}^k\bigr)
\leqslant\sum_{x\in\T_\alpha^a}\nu_\alpha(\up x)\sum_{y\in\M^a_{\alpha\cdot x}}\len{xy}^kf_{\alpha\cdot x}(y)
\end{align*}
Next:~$\len{xy}^k=(\len x+\len y)^k\leqslant 2^k(\len x^k+\len y^k)$.
Hence, if~$M_k$ is a common bound of the series in~\eqref{eq:48} for~$\alpha$ ranging over~$X$, one has:
\begin{gather*}
\esp_\alpha\bigl(\len {U_\alpha^a}^k\bigr)\leqslant 2^k \sum_{x\in\T^a_\alpha}\nu_\alpha(\up x)\bigl(M_0 \len x^k+M_k\bigr)
\leqslant 2^k(M_k+ M_0 \esp_\alpha(\len{T_\alpha^a}^k)\bigr)<\infty
\end{gather*}
using Lemma~\ref{lem:5}, which completes the proof.
\end{proof}

Before proceeding with the proof of Theorem~\ref{thr:6}, we need one more lemma.
For this, assume given a letter~$a\in\Sigma$, and let~$\omega\in\BM_\alpha$ be an infinite trajectory such that~$(T_\alpha^a)^n(\omega)$ are finite for all~$n\geqslant0$.
Consider the following sequences:
\begin{align}
  \label{eq:67}
x_j&=(T_\alpha^a)^j(\omega)  &\xi_j&=(x_j)^{-1}\omega& \alpha_j&=\alpha\cdot x_j&u_j=U^a_{\alpha_j}(\xi_j)
\end{align}
Seen as random variables, it follows from Lemma~\ref{lem:5} that  all the sequences in~\eqref{eq:67} are~$\nu_\alpha$-\as\ well defined.

\begin{lemma}
\label{lem:9}
  For~$\nu_\alpha$-\as\ every~$\omega\in\BM_\alpha$, there exists an integer~$J$ such that:
  \begin{gather}
    \label{eq:76}
    \forall j\geqslant0\quad\bigl( j\geqslant J\implies |u_j|\leqslant\sqrt \jmath\,\bigr)
  \end{gather}
\end{lemma}

\begin{proof}
The random variable~$\xi_j$ can be written as~$\xi_j(\omega)=\theta_{(T^a_\alpha)^j(\omega)}(\omega)$.
It follows thus from the Strong Markov property~\eqref{eq:43} on the one hand, and from Lemma~\ref{lem:9} on the other hand, that~$\esp_\alpha(\len {u_j}^3)\leqslant M$ for some constant~$M$, for all~$\alpha\in X$ and for all~$j\geqslant0$.
Hence, by Markov inequality:
\begin{align*}
  \sum_{j>0}\nu_\alpha\bigl(|u_j|>\sqrt \jmath\bigr)\leqslant M\sum_{j>0}\frac1{j^{3/2}}<\infty
\end{align*}
which proves~\eqref{eq:76} \emph{via} the Borel-Cantelli lemma.
\end{proof}

\begin{proof}[Proof of Theorem~\ref{thr:6}.]
\emph{Point~\ref{item:14}.}\quad  We only consider the case where~$\varphi$ is non-decreasing and sub-additive, since the case where~$\varphi$ is additive is more straightforward.
Pick an arbitrary letter~$a\in\Sigma$ and let~$\T^a$ be the set of traces of the form~$T^a(\omega)$ for some infinite trajectory~$\omega$.
Let~$\ell$ be the limit of the ergodic means~$M\varphi((T^a_\alpha)^n)$, which exists according to Lemma~\ref{lem:6}.
Let also~$\Omega_\alpha$ be a subset of~$\BM_\alpha$ with~$\nu_\alpha(\Omega_\alpha)=1$ and for which~\eqref{eq:76} holds for all~$\omega\in\Omega$.

Let~$\omega\in\Omega_\alpha$.
We prove that~$\lim_{x\in\M_\alpha,\ x\to\omega}M\varphi(x)=\ell$.
Consider the sequences~$\seq xj$,~$\seq\xi j$,~$\seq\alpha j$ and~$\seq uj$ introduced in~\eqref{eq:67}.
Let~$L$ be a constant such that:
\begin{align}
  \label{eq:77}
  \forall (\beta,b)\in X\times\Sigma\quad\varphi_\beta(b)\leqslant L
\end{align}
The subadditivity of~$\varphi$ implies in particular:
\begin{gather}
\label{eq:58}
  \forall \beta\in X\quad\forall x\in\M_\beta\quad \varphi_\beta(x)\leqslant L\len x
\end{gather}

Fix~$\varepsilon>0$ and let~$N$ be an integer such that, for all~$n\geqslant N$:
\begin{align}
  \label{eq:78}
  |M\varphi(x_n)-\ell|&<\frac\ve2&
                             \len {u_n}&\leqslant\sqrt n&
                                                     \frac L{\sqrt N}&<\frac\ve4
\end{align}

Let~$y\in\M_\alpha$ be such that~$x_N\leqslant y\leqslant\omega$.
We prove that~$|M\varphi(y)-\ell|<\ve$.

Let~$j$ be the number of occurrences of~$a$ in~$y$.
Then~$y=x_jz$ where~$z\in\M^a$, and~$j\geqslant N$ since~$x_N$ has~$N$ occurrences of~$a$ and~$x_N\leqslant y$.
Furthermore,~$z\leqslant u_j$ by the definition of~$u_j$.
\begin{align*}
M\varphi(y)-M\varphi(x_j)
=
  \frac{\varphi_\alpha(x_jz)}{\len{x_j}+\len z}-\frac{\varphi_\alpha(x_j)}{\len{x_j}}=
  \underbrace{\frac{\varphi_\alpha(x_jz)-\varphi(x_j)}{\len{x_j}+\len z}}_A-
  \underbrace{\frac{\len z}{\len{x_j}+\len z}\cdot\frac{\varphi_\alpha(x_j)}{\len{x_j}}}_B
\end{align*}

For controlling the term~$A$, we use the fact that~$\varphi$ is sub-additive and non-decreasing, then~$z\leqslant u_j$ then~\eqref{eq:58} and~$\len{u_j}\leqslant\sqrt \jmath$ from~\eqref{eq:78}, and finally~$\len{x_j}\geqslant j$ since~$x_j$ has at least its~$j$ occurrences of~$a$ to get:
\begin{align*}
  0\leqslant A&\leqslant\frac{\varphi_{\alpha\cdot x_j}(z)}{\len {x_j}}
\leqslant\frac{\varphi_{\alpha\cdot x_j}(u_j)}{\len{x_j}}
\leqslant L\frac {\sqrt \jmath}{\len{x_j}}\leqslant L\frac1{\sqrt \jmath}<\frac\ve4
\end{align*}

For controlling the term~$B$, we use the bound~$L$ for~$M\varphi(x_j)$ from~\eqref{eq:77},~$\len z\leqslant\len{u_j}\leqslant\sqrt \jmath$ and~$\len{x_j}\geqslant j$ to get~$0\leqslant -B\leqslant L/\sqrt \jmath<\ve/4$.

Combined with the previous estimate for~$A$, we obtain~$|M\varphi(y)-M\varphi(x_j)|<\ve/2$, yielding finally~$|\M\varphi(y)-\ell|<\ve$, which was to be proved.

\emph{Point~\ref{item:15}.}\quad Assume that~$\varphi$ is additive.
Let~$(\alpha_i,C_{i+1})_{i\geqslant0}$ be the Markov chain of state-and-cliques.
We compute the ergodic means~$M\varphi$ using the exhaustive sequence~$Y_n=C_1\cdots C_n$, yielding:
\begin{gather}
  \label{eq:70}
  M\varphi(Y_n)=\frac n{\len {C_1}+\dots+\len{C_n}}\cdot\frac{\varphi_{\alpha_0}(C_1)+\ldots+\varphi_{\alpha_{n-1}}(C_n)}{n}
\end{gather}

Let~$J$ be a basic component of \DSCP.
Let~$(\beta_0,d_0)$ be a state-and-clique of~$J$, and let~$x_0$ be a protection of~$(\beta_0,d_0)$.
All infinite trajectories~$\omega\in\up {x_0}$ have their cliques in~$J$ since~$J$ is final in \DSCP\ and since~$C_1(\omega)=d_0$.
On~$\up{x_0}$, the ergodic means~\eqref{eq:70} converge~$\nu_{\beta_0}$-\as\ toward the constant~$k_\varphi$ defined in~\eqref{eq:69}; but the~$\nu_\alpha$-\as\ limit of~$M\varphi$ is independent of~$\alpha$ according to the point~\ref{item:14} already proved, whence the result.
\end{proof}

\subsection{Computing the speedup}
\label{sec:aver-conc-rate}

As a non-negative, non-decreasing and sub-additive function on traces, the height function is a test function.
Hence, Theorem~\ref{thr:6} allows to introduce the following definition.

\begin{definition}
  Let\/~$\SS=(\M,X)$ be an irreducible concurrent system equipped with its uniform measure~$\nu$.
  The \emph{speedup} of~$\SS$ is the limit:
  \begin{gather}
    s=\lim \frac{\len x}{\height x}
  \end{gather}
as~$x\to\omega\in\BM_\alpha$, which is a constant independent of~$\alpha\in X$ and~$\nu_\alpha$-\as\ independent of~$\omega$.
\end{definition}

The speedup is an average rate of parallelism of concurrent executions; the highest the speedup, the more there is concurrency.
An effective way of computing it is as follows.

\begin{proposition}
  \label{prop:2}
  Let~$\SS=(\M,X)$ be an irreducible concurrent system equipped with a Markov measure.
Let~$J$ be a final component of\/ \DSCP.
Let~$Q$ be transition matrix of the Markov chain of state-and-cliques, restricted to~$J$, and let~$\pi_J$ be the unique invariant probability distribution of~$Q$.
Then the speedup is:
  \begin{gather}
    \label{eq:60}
s=    \sum_{(\alpha,d)\in J}\len d\,\pi_J\bigl((\alpha,d)\bigr)
  \end{gather}
and this quantity is independent of~$J$.
\end{proposition}

\begin{proof}
Let~$\varphi_\alpha(x)=\height x$.
We compute the limit of~$M\varphi$ using the exhaustive sequence given by~$Y_j=C_1\cdots C_j$:
\begin{align}
  \label{eq:66}
M\varphi(Y_j)&=  \frac{\len {Y_j}}{\height{Y_j}}=\frac{\len {Y_j}}j=\frac1j\sum_{k=1}^j\len{C_k}
\end{align}

Let~$J$ be a basic component of \DSCP.
We proceed along the same line of proof as in the proof of Theorem~\ref{thr:6}, point~\ref{item:15}.
Let~$(\alpha_0,c_0)$ be a state-and-clique of~$J$, and let~$x_0\in\M_{\alpha_0}$ be a protection of~$(\alpha_0,c_0)$.
All infinite trajectories~$\omega\in\up x_0$ have their cliques in~$J$ since~$J$ is final and~$C_1(\omega)=(\alpha_0,x_0)$.
But, on~$\up x_0$, the ergodic means~\eqref{eq:66} converge \as\ toward~$s$ defined in~\eqref{eq:60}.

Since the limit of ergodic means is \as\ constant according to Theorem~\ref{thr:6}, the result follows.
\end{proof}

\begin{remark}
  \label{rem:3}
  It is a combinatorial result that the quantity~$s$ defined in~\eqref{eq:60} is independent of the final component~$J$.
Yet, it is obtained with probabilistic arguments.

\end{remark}

\section{Additional examples}
\label{sec:applications}

\subsection{About the basic components of the \DSC}
\label{sec:an-example-with}

What does it represent for a concurrent system that its \DSC\ has several basic components? Without concurrency, hence if~$\Sigma$ is a free monoid, then the \DSC\ is strongly connected and coincides with the \DSCP; the presence of several basic components of \DSC\ is thus closely related to the concurrency features of the model.

Consider an infinite trajectory~$\omega\in\BM_\alpha$.
Almost surely, the cliques of~$\omega$ will fall within one of the basic components of~$\DSCP$.
Now consider a finite prefix~$x\leqslant\omega$ and~$\xi=x^{-1}\omega$.
Then the cliques of~$\xi$ will also fall with probability~$1$ within a basic components of~$\DSCP$, but maybe within an other one, as the following result shows.

\begin{proposition}
\label{prop:1}
  Let~$\alpha\in X$ be an initial state.
Then for~$\nu_\alpha$-almost every~$\omega\in\BM_\alpha$ and for every basic component~$J$ of\/~$\DSCP$, there is a prefix~$x\leqslant\omega$ such that all cliques of~$x^{-1}\omega$ belong to~$J$.
\end{proposition}

\begin{proof}
Let~$(\beta,c)$ be a state-and-clique of~$J$, and let~$y\in\M_\beta$ be one of its protections.
Then, working with the stopping times as in~\S~\ref{sec:ergod-prop-irred}, it is routine to check that for~$\nu_\alpha$-\as\ every~$\omega\in\BM_\alpha$, there is a prefix~$x\leqslant\omega$ such that~$\alpha\cdot x=\beta$ and~$y\leqslant x^{-1}\omega$ (note:
that does not mean that the cliques~$(c_i)_{i\geqslant1}$ of~$\omega$ satisfy~$c_1\ldots c_n=x$ for some integer~$n\geqslant1$).
Then~$\xi=x^{-1}\omega$ has~$(\beta,c)$ as its first state-and-clique, and then all the state-and-cliques of~$(\beta,\xi)$ belong to~$J$ since~$J$ is final in \DSCP.
\end{proof}

Hence, the presence of several basic components of~$F$ reveals an artefact of the combinatorial encoding by state-and-cliques;
indeed, a same infinite trajectory is essentially ``seen'' from within all the basic components, according to the ``moment'' one decides to start the encoding.

\begin{figure}[t]
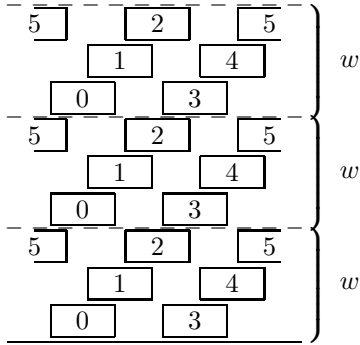

  \centering
\begin{gather*}
\xy
<.1em,0em>:
0="B",
"B"+0="G",
"G"+(12,6)*{0},
"G";"G"+(24,0)**@{-};"G"+(24,12)**@{-};"G"+(0,12)**@{-};"G"**@{-},
"B"+(42,0)="G",
         "G"+(12,6)*{3},
"G";"G"+(24,0)**@{-};"G"+(24,12)**@{-};"G"+(0,12)**@{-};"G"**@{-},
"B"+(14,14)="G";         
         "G"+(12,6)*{1},
"G";"G"+(24,0)**@{-};"G"+(24,12)**@{-};"G"+(0,12)**@{-};"G"**@{-},
"B"+(56,14)="G";         
         "G"+(12,6)*{4},
"G";"G"+(24,0)**@{-};"G"+(24,12)**@{-};"G"+(0,12)**@{-};"G"**@{-},
"B"+(28,28)="G";         
         "G"+(12,6)*{2},
"G";"G"+(24,0)**@{-};"G"+(24,12)**@{-};"G"+(0,12)**@{-};"G"**@{-},
"B"+(70,28)="G";         
         "G"+(12,6)*{5},
         "G"+(16,0);"G"**@{-};"G"+(0,12)**@{-};"G"+(16,12)**@{-},
         "B"+(-6,28)="G";
         "G"+(0,6)*{5},
         "G";"G"+(12,0)**@{-};"G"+(12,12)**@{-};"G"+(0,12)**@{-},
(0,42)="B",
"B"+0="G",
"G"+(12,6)*{0},
"G";"G"+(24,0)**@{-};"G"+(24,12)**@{-};"G"+(0,12)**@{-};"G"**@{-},
"B"+(42,0)="G",
         "G"+(12,6)*{3},
"G";"G"+(24,0)**@{-};"G"+(24,12)**@{-};"G"+(0,12)**@{-};"G"**@{-},
"B"+(14,14)="G";         
         "G"+(12,6)*{1},
"G";"G"+(24,0)**@{-};"G"+(24,12)**@{-};"G"+(0,12)**@{-};"G"**@{-},
"B"+(56,14)="G";         
         "G"+(12,6)*{4},
"G";"G"+(24,0)**@{-};"G"+(24,12)**@{-};"G"+(0,12)**@{-};"G"**@{-},
"B"+(28,28)="G";         
         "G"+(12,6)*{2},
"G";"G"+(24,0)**@{-};"G"+(24,12)**@{-};"G"+(0,12)**@{-};"G"**@{-},
"B"+(70,28)="G";         
         "G"+(12,6)*{5},
         "G"+(16,0);"G"**@{-};"G"+(0,12)**@{-};"G"+(16,12)**@{-},
         "B"+(-6,28)="G";
         "G"+(0,6)*{5},
         "G";"G"+(12,0)**@{-};"G"+(12,12)**@{-};"G"+(0,12)**@{-},
(0,84)="B",
"B"+0="G",
"G"+(12,6)*{0},
"G";"G"+(24,0)**@{-};"G"+(24,12)**@{-};"G"+(0,12)**@{-};"G"**@{-},
"B"+(42,0)="G",
         "G"+(12,6)*{3},
"G";"G"+(24,0)**@{-};"G"+(24,12)**@{-};"G"+(0,12)**@{-};"G"**@{-},
"B"+(14,14)="G";         
         "G"+(12,6)*{1},
"G";"G"+(24,0)**@{-};"G"+(24,12)**@{-};"G"+(0,12)**@{-};"G"**@{-},
"B"+(56,14)="G";         
         "G"+(12,6)*{4},
"G";"G"+(24,0)**@{-};"G"+(24,12)**@{-};"G"+(0,12)**@{-};"G"**@{-},
"B"+(28,28)="G";         
         "G"+(12,6)*{2},
"G";"G"+(24,0)**@{-};"G"+(24,12)**@{-};"G"+(0,12)**@{-};"G"**@{-},
"B"+(70,28)="G";         
         "G"+(12,6)*{5},
         "G"+(16,0);"G"**@{-};"G"+(0,12)**@{-};"G"+(16,12)**@{-},
         "B"+(-6,28)="G";
         "G"+(0,6)*{5},
         "G";"G"+(12,0)**@{-};"G"+(12,12)**@{-};"G"+(0,12)**@{-},
(-16,-2);(94,-2)**@{-},
(-16,41);(94,41)**@{--},
(-16,83);(94,83)**@{--},
(-16,125);(94,125)**@{--},
(100,0)="G";
"G"+(12,20)*{w},
"G";"G"+(0,41)**\frm{)},
(100,42)="G";
"G"+(12,20)*{w},
"G";"G"+(0,41)**\frm{)},
(100,84)="G";
"G"+(12,20)*{w},
"G";"G"+(0,41)**\frm{)},
\endxy
\end{gather*}
  \caption{The trace~$w^3$ with~$w=031425$ represented as a heap of pieces with a cyclic base (identify half-pieces labeled~`$5$' with the same altitude)}
  \label{fig:qwpokqwdqmqaszx}
\end{figure}

\paragraph{An example with several non-isomorphic basic components for the \DSCP.}

Consider the irreducible trace monoid~$\M = \M(\Sigma,I)$ where~$\Sigma = \mathbb{Z}/6\mathbb{Z}$ and~$I = \{(a,b) \tq a-b \notin \{-1,0,1\}\}$.
Let~$w$ be the trace~$031425$, and let~$\sim$ be the equivalence relation on traces~$x,y \in \M$ defined by~$x \sim y$ when there exist integers~$k,\ell \geqslant 0$ such that~$w^k x = w^\ell y$;
each equivalence class contains a unique trace~$x$ such that~$w \not\leqslant x$.

Let~$\W$ be the set of traces~$x$ that left-divide some power of~$w$, \ie,~$\W = \{x \in \M \tq \exists k \geqslant 0 \quad x \leqslant w^k\}$.
Each equivalence class for~$\sim$ that contains an element of~$\W$ is included in~$\W$; hence, let~$X$ be the set of these equivalence classes, \ie,~$X = \mathcal{W} / \sim$.

The set~$X$ is finite.
More precisely, for every~$x\in\W$, there exists~$x'\in\W$ such that~$x\sim x'$ and~$x'\leqslant w^2$.
To prove it, let~$x\in\W$ such that~$x\leqslant w^3$.
Without loss of generality, we may assume that~$w\not\leqslant x$.
Then either~$\height x\leqslant 6$, and then~$x\leqslant w^2$, or~$\height x>6$.
In the latter case,~$x$~must contain the occurrence of~$0$ or of~$3$ of the seventh layer of~$w^2$, as depicted on Fig.~\ref{fig:qwpokqwdqmqaszx}.
But then it is apparent on Fig.~\ref{fig:qwpokqwdqmqaszx} that~$w\leqslant x$, a contradiction; hence~$x\leqslant w^2$, as claimed.

Moreover, the relation~$\sim$ is stable by right product, \ie, if~$x \sim y$, then~$x z \sim y z$.
Consequently, denoting by~$[x]$ the equivalence class of a trace~$x$, we can let~$\M$ act on~$X$ by setting~$[x] \cdot a = [x a]$ when~$x a \in \W$ and~$[x] \cdot a = \bot$ when~$x a \notin \W$.
This yields the concurrent system~$\mathcal{S} = (\M,X)$ depicted on Fig.~\ref{fig:kljaskh}, each class being represented by its unique element of smallest length.

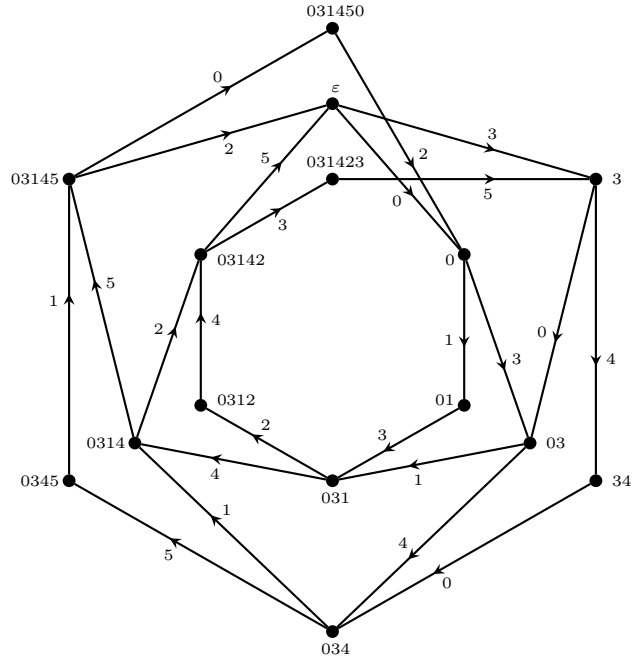
\begin{figure}
\centering
\begin{tikzpicture}[decoration={markings,mark=at position 0.625 with {\arrow{stealth} \node at (-0.1,\p 0.2) {\tiny \c};}}]
\foreach[count=\x] \a/\r/\l/\d/\y in {0/3/\varepsilon/south/1,60/2/0/east/-2,60/4/3/west/0,120/2/01/east/2,120/4/34/west/0,
120/3/03/west/0,180/2/031/north/-1,180/4/034/north/-1,240/2/0312/west/2,240/4/0345/east/0,
240/3/0314/east/0,300/2/03142/west/-2,300/4/03145/east/0,0/2/031423/south/1,0/4/031450/south/1}{
 \node[circle,fill=black,draw=black,thick,inner sep=0pt,minimum size=4pt] (n\x) at (90-\a:\r) {};
 \node[anchor=\d,yshift=\y pt] at (90-\a:\r) {\tiny~$\l$};
}
\foreach \a/\b/\c/\p in {1/2/0/-,1/3/3/+,2/6/3/+,3/6/0/-,6/7/1/+,6/8/4/-,7/11/4/+,8/11/1/-,11/12/2/+,11/13/5/-,12/1/5/+,13/1/2/-,
2/4/1/-,4/7/3/-,7/9/2/-,9/12/4/-,12/14/3/-,14/3/5/-,
3/5/4/+,5/8/0/+,8/10/5/+,10/13/1/+,13/15/0/+,15/2/2/+}{
 \draw[thick,postaction={decorate}] (n\a) -- (n\b);
}
\end{tikzpicture}
\caption{Action of~$(\M,X)$ where each class in~$X$ is represented by its unique element~$\leqslant w^2$}
\label{fig:kljaskh}
\end{figure}

From this representation, one sees readily that~$\mathcal{S}$ satisfies all the criteria for being irreducible and that, for each state~$\alpha \in X$, the only clique~$c \in \M$ for which the state-and-clique~$(\alpha,c)$ admits a protection is the clique containing all the letters~$a \in \Sigma$ such that~$\alpha \cdot a \neq \bot$.
Thus, its \DSCP{} is represented on Fig.~\ref{fig:pokjdwfn}: its basic components are cycles of lengths~$3$ and~$6$.

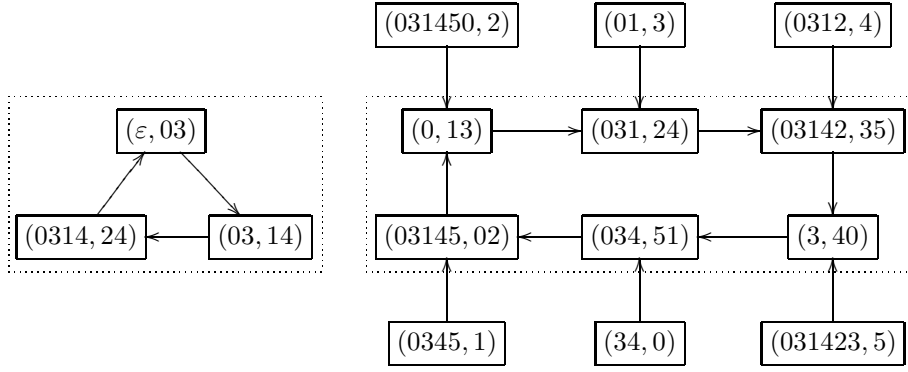
\begin{figure}
\begin{gather*}
  \xymatrix{
&&    *+[F]{(031450,2)}\ar[d]&*+[F]{(01,3)}\ar[d]&*+[F]{(0312,4)}\ar[d]\\
 \save[]+<3em,0em>*+[F]{(\ve,03)}\ar@{<-}[d]\ar[dr]\restore&&    *+[F]{(0,13)}\ar[r]&*+[F]{(031,24)}\ar[r]&*+[F]{(03142,35)}\ar[d]\\
 *+[F]{(0314,24)}      &*+[F]{(03,14)}\ar[l]&    *+[F]{(03145,02)}\ar[u]&*+[F]{(034,51)}\ar[l]&*+[F]{(3,40)}\ar[l]\\
&&    *+[F]{(0345,1)}\ar[u]&*+[F]{(34,0)}\ar[u]&*+[F]{(031423,5)}\ar[u]
\save{"3,3"+<-3em,-1.3em>}.{"2,5"+<3em,1.3em>}*[F.]\frm{}\restore
\save{"2,1"+<-2.7em,1.3em>}.{"3,2"+<2.3em,-1.3em>}*\frm{.}\restore  
  }
\end{gather*}

  \caption{The \DSCP\ with two basic components (framed)}
  \label{fig:pokjdwfn}
\end{figure}

\begin{remark}
  \label{rem:2}
  On this example, the probabilistic structure is trivial since, for every state~$\alpha$, there is a unique infinite trajectory starting from~$\alpha$.
\end{remark}

\subsection{Doubled trace monoids}
\label{sec:doubled-trace-monoid}

The structure of a given concurrent system can give short ways of describing the uniform measure.
Doubled trace monoids are an example of this kind.

Consider an irreducible trace monoid~$\M=\M(\Sigma,I)$, put~$X=\C$ the set of cliques of~$\M$, and define a mapping~$X\times\Sigma\to X\cup\{\bot\},\quad (\gamma,a)\mapsto\gamma\cdot a$ as follows:
\begin{gather*}
\gamma\cdot a=
\begin{cases}
\gamma\cup\{a\},&\text{if~$a\notin\gamma$ and~$\gamma\cup\{a\}\in\C$}\\
\gamma\setminus\{a\},&\text{if~$a\in\gamma$}\\
\bot,&\text{otherwise}
\end{cases}
\end{gather*}
This mapping extends in a unique way to an action~$X\times\M\to X\cup\{\bot\}$ making~$(\M,X)$ an irreducible concurrent system.

For instance, if~$\Sigma=\{a_0,\ldots,a_4\}$ with~$a_ia_j=a_ja_i$ when~$i-j\neq\pm1\mod 5$, the concurrent system corresponds to the so-called ``five dining philosophers'' model introduced by Dijkstra in concurrency theory~\cite{dijkstra71}.

The probabilistic valuation~$f_\alpha(x)=\rho^{\len x}\Delta(\alpha,\alpha\cdot x)$ associated with the uniform measure has a simple expression for these models:
\begin{compactitem}
\item[(\dag)] Let~$r$ be the root of smallest modulus of the Möbius polynomial of the monoid~$\M$.
Then~$\rho=\sqrt r$ and~$\Delta(\alpha,\beta)=r^{\frac12(\len\beta-\len\alpha)}$.
\end{compactitem}

\begin{proof}
Let~$\nu$ be the uniform measure of the concurrent system, and let~$f$ be the corresponding probabilistic valuation.
Let~$\gamma\in X$, hence~$\gamma$ is a clique of~$\M$.
It follows from Corollary~\ref{cor:7} that~$\nu_\gamma(\up\gamma)=1$;
indeed, otherwise some letter would be never played.
Therefore~$f_\gamma(\gamma)=1$, and since~$\gamma\cdot\gamma=\ve$ in the action, it yields~$\rho^{\len\gamma}\Delta(\gamma,\ve)=1$ and therefore, using the cocycle property of~$\Delta$:
$\Delta(\ve,\gamma)=\rho^{\len\gamma}$, and finally~$\Delta(\gamma,\gamma')=\rho^{\len{\gamma'}-\len\gamma}$.
As a result,~$f_\ve(\gamma)=\rho^{2\len\gamma}$.

It suffices now to prove that~$\rho=\sqrt r$.
Let~$\psi$ be the Möbius transform of the uniform valuation~$\varphi(x)=\rho^{2\len x}$ of the monoid~$\M$.
Given the form already found for the cocycle~$\Delta$, the condition~\eqref{eq:11} for~$f$ being probabilistic, used at state~$\ve$, shows that~$\rho^2$ satisfies:
\begin{align*}
\psi(\ve)&=0
&\forall \gamma\in\Cstar\quad\psi(\gamma)&>0
\end{align*}
But only the root of smallest modulus~$r$ has this property (short proof:
by the uniqueness of Theorem~\ref{thr:2}, since it allows to construct a uniform measure on~$\BM$).
Hence,~$r=\rho^2$, which completes the proof of~(\dag).
\end{proof}

\printbibliography
\end{document}